\documentclass[11pt, a4paper]{amsart}

\usepackage{dsfont}

\usepackage{amsfonts,amsmath,amssymb, amscd,fullpage}
\usepackage{stmaryrd}
\usepackage[all]{xy}
\usepackage{enumerate} 
\usepackage{tikz}

\newtheorem{theorem}{Theorem}[section]
\newtheorem{lem}[theorem]{Lemma}
\newtheorem{prop}[theorem]{Proposition}
\newtheorem{cor}[theorem]{Corollary}
\newtheorem{example}[theorem]{Example}
\newtheorem{remark}[theorem]{Remark}

\theoremstyle{definition}
\newtheorem{defin}[theorem]{Definition}

\newcommand\pf{\begin{proof}}
\newcommand\epf{\end{proof}}

\newcommand{\ds}{\displaystyle}

\newcommand{\Id}{\mathrm{Id}}
\newcommand{\Z}{\mathbb{Z}}
\newcommand{\N}{\mathbb{N}}

\newcommand{\C}{\mathbb{C}}

\newcommand{\id}{\mathrm{id}}
\newcommand{\Vect}{\mathrm{Vect}}
\newcommand{\Hom}{\mathrm{Hom}}
\newcommand{\Aut}{\mathrm{Aut}}
\newcommand{\TX}{\mathcal{T}^X}
\newcommand{\TG}{\mathcal{T}^G}

\newcommand{\M}{\mathcal{M}}

\numberwithin{equation}{subsection}

\title{Study of Quantum symmetries for vertex-transitive graphs using intertwinner spaces.}

\author{Arthur Chassaniol}
\address{}

\begin{document}

\maketitle

\textbf{Abstract.} We study the quantum automorphism group of vertex-transitive graphs using intertwinner spaces of the magic unitary matrix associated to this quantum subgroups of $S_n^+$. We also give some applications to quantum symmetries of circulant graphs.\bigskip\bigskip\smallskip

\section{Introduction}

A quantum permutation group on $n$ points is a compact quantum group acting faithfully on the classical space consisting of $n$ points. The following facts were discovered by Wang \cite{Wang}.

\begin{enumerate}
 \item There exists a largest quantum permutation group on $n$ points,  denoted $S_n^+$, and called \textsl{the} quantum permutation groups on $n$ points.
\item The quantum group $S_n^+$ is infinite-dimensional if $n \geq 4$, and hence in particular an infinite compact quantum group can act faithfully on a finite classical space. 
\end{enumerate}

Very soon after Wang's paper \cite{Wang}, the representation theory of $S_n^+$ was described  by Banica \cite{Sygeco}: it is similar to the one of $SO(3)$ and can be described using tensor categories of non-crossing partitions. This description, further axiomatized and generalized by Banica-Speicher \cite{bs}, led later to spectacular connections with free probability theory, see e.g. \cite{ks}.

The next natural question was the following one: does $S_n^+$ have many non-classical quantum subgroups, or is it isolated as an infinite quantum group acting faithfully on a finite classical space?

In order to find quantum subgroups of $S_n^+$, the quantum automorphism group of a finite graph was defined in \cite{leta,NeTu}. This construction indeed produced many examples of non-classical quantum permutation groups, answering positively to the above question.  
The known results on the computation of quantum symmetry groups of graphs are summarized in \cite{BaBi11} where the description of the quantum symmetry group of vertex-transitive graphs of small order (up to $11$) is given with an exception for the Petersen graph. The quantum symetries of Petersen graph was study in \cite{Sim} where the autor proves that Petersen graph has no quantum symmetry.

\medskip

The present paper is a contribution to the study of quantum automorphism groups of finite graphs: we study some way to prove that a vertex-transitiv graph has no quantum symmetries, with the study of associated intertwinners spaces (as in \cite{bs}). With our results we can better understand the quantum symmetries of some circulant graph, as $C_{13}(3,4)$.

\medskip

 The paper is organized as follows. Section 2 is preliminary section: we recall some basic facts  about compact quantum groups, quantum permutation groups and quantum automorphism group of finite graphs. Section 3 is devoted to intertwinner spaces for classical subgroups of $S_n$. In Sections 4 we study one particular intertwinnner space associated to vertex-transitive graphs. In section 5 we extend the study to the quantum case, and introduce the notion of $\mathcal{B}$-clos graphs. To finish the sections 6,7,8 and 9 are devoted to applications of results from section $5$.

\newpage

\section{Compact quantum groups and quantum automorphism group of finite graphs}

We first recall some basic facts concerning compact quantum groups. The books \cite{KliSch,NeTu} are convenient references for this topic, and all the definitions we omit can be found there. All algebras in this paper will be unital as well as all algebra morphisms, and $\otimes$ will denote the minimal tensor product of $C^*$-algebras as well as the algebraic tensor product, this should cause no confusion.

\begin{defin} A \emph{Woronowicz algebra} is a $C^*$-algebra $A$ endowed with a $*$-morphism $\Delta:A\to A\otimes A$ satisfying the coassociativity condition and the cancellation law
$$\overline{\Delta(A)(A\otimes 1)}=A\otimes A=\overline{\Delta(A)(1\otimes A)}$$
The morphism $\Delta$ is called the comultiplication of $A$.
\end{defin}

The category of Woronowicz algebras is defined in the obvious way. A commutative Woronowicz algebra is isomorphic with $C(G)$, the algebra of continuous functions on a compact group $G$, unique up to isomorphism, and the category of \emph{compact quantum groups} is defined to be the category dual to the category of Woronowicz algebras. Hence to any Woronowicz algebra $A$ corresponds a unique compact quantum group $G$ according to the heuristic notation $A=C(G)$.

Woronowicz's original definition for matrix compact quantum groups \cite{Woro1} is still the most useful to produce concrete examples, and we have the following fundamental result \cite{Woro2}.

\begin{theorem} Let $A$ be a $C^*$-algebra endowed with a $*$-morphism $\Delta: A\to A\otimes A$. Then $A$ is a Woronowicz algebra if and only if there exists a family of unitary matrices $(u_{\lambda})_{\lambda\in\Lambda}\in M_{d_{\lambda}}(A)$ satisfying the following three conditions:
\begin{enumerate}
\item The $*$-subalgebra $A_0$ generated by the entries $u_{ij}^{\lambda}$ of the matrices $(u^{\lambda})_{\lambda\in\Lambda}$ is dense in $A$.
\item For $\lambda\in\Lambda$ and $i,j\in\{1,\dots, d_{\lambda}\}$, one has $\Delta(u_{ij}^{\lambda})=\sum_{k=1}^{d_{\lambda}}{u_{ik}^{\lambda}\otimes u_{kj}^{\lambda}}$.
\item For $\lambda\in\Lambda$, the transpose matrix $(u^{\lambda})^t$ is invertible.
\end{enumerate}
\end{theorem}

In fact, the $*$-algebra $A_0$ in the theorem is canonically defined, and is what is now called a compact Hopf algebra: a Hopf $*$-algebra having all its finite-dimensional comodules equivalent to unitary ones (see \cite{KliSch, NeTu}). The counit and antipode of $A_0$, denoted, respectively, $\epsilon$ and $S$, are referred to as the counit and antipode of $A$. The Hopf $*$-algebra $A_0$ is called the \emph{algebra of representation functions} on the compact quantum group $G$ dual to $A$, with another heuristic notation $A_0=\mathcal{O}(G)$. 

Conversely, starting from a compact Hopf algebra, the universal $C^*$-completion yields a Woronowicz algebra in the above sense: see \cite{KliSch, NeTu}. In fact there are possibly several different $C^*$-norms on $A_0$, but we will not be concerned with this question.\smallskip

As usual, a (compact) quantum subgroup $H\subset G$ corresponds to a surjective Woronowicz algebra morphism $C(G)\to C(H)$, or to a  surjective Hopf $*$-algebra morphism $\mathcal{O}(G)\to\mathcal{O}(H)$.\smallskip                          

We refer the reader to \cite{KliSch, NeTu} for large classes of examples, including $q$-deformations of classical compact Lie groups. In the present paper, we will be interested in the following fundamental example, due to Wang \cite{Wang}. First we need some terminology. A matrix $u\in M_n(A)$ is said to be orthogonal if $u=\bar{u}$ and $uu^t=I_n=u^tu$. A matrix $u$ is said to be magic unitary if all its entries are projections, all distinct elements  of a same row or same column are orthogonal, and sums of rows and columns are equal to $1$. A magic unitary matrix is orthogonal.

\begin{defin} The $C^*$-algebra $A_s(n)$ is defined to be the universal $C^*$-algebra generated by variables $(u_{ij})_{1\le i,j\le n}$, with relations making $u=(u_{ij})$ a magic unitary matrix.

The $C^*$-algebra $A_s(n)$ admits a Woronowicz algebra structure given by
$$\Delta(u_{ij}) =\sum_{k=1}^n{u_{ik}\otimes u_{kj}},\ \ \ \ \ \ \epsilon(u_{ij}) =\delta_{ij},\ \ \ \ \ \ S(u_{ij}) =u_{ji}$$

The associated compact quantum group is denoted by $S_n^+$, i.e.
$$A_s(n)=C(S_n^+)$$
\end{defin}

\begin{defin} A \emph{quantum permutation algebra} is a Woronowicz algebra quotient of $A_s(n)$ for some $n$. Equivalently, it is a Woronowicz algebra generated by the coefficients of a magic unitary matrix.
\end{defin}


We now come to quantum group actions, studied e.g. in \cite{Podl}. They correspond to Woronowicz algebra coactions. Recall that if $B$ is a $C^*$-algebra, a (right) \emph{coaction} of Woronowicz algebra $A$ on $B$ is a $*$-homomorphism $\alpha: B\to B\otimes A$ satisfying the coassociativity condition and 
$$\overline{\alpha(B)(1\otimes A)}=B\otimes A$$

Wang has studied quantum groups actions on finite-dimensional $C^*$-algebras in \cite{Wang}, where the following result is proved.

\begin{theorem} The Woronowicz algebra $A_s(n)$ is the universal Woronowicz algebra coacting on $\C^n$, and is infinite-dimensional if $n\ge 4$.
\end{theorem}

The coaction is constructed in the following manner. Let $e_1,\dots,e_n$ be the canonical basis of $\C^n$. Then the coaction $\alpha: \C^n\to\C^n\otimes A_s(n)$ is defined by the formula
$$\alpha(e_i)=\sum_{j=1}^n{e_j\otimes u_{ji}}$$

We refer the reader to \cite{Wang} for the precise meaning of universality in the theorem, but roughly speaking this means that $S_n^+$ is the largest compact quantum group acting on $n$ points, and deserves to be called the \emph{quantum permutation group on n points}.\smallskip

Equivalently, Wang's theorem states that any Woronowicz algebra coacting faithfully on $\C^n$ is a quotient of the Woronowicz algebra $A_s(n)$, and shows that quantum groups acting on $n$ points correspond to Woronowicz algebra quotient of $A_s(n)$, and hence to quantum permutation algebras. In particular, there is a surjective Woronowicz algebra morphism $A_s(n)\to C(S_n)$, yielding a quantum group embedding $S_n\subset S_n^+$. More directly, the existence of the surjective morphism $A_s(n)\to C(S_n)$ follows from the fact that $C(S_n)$ is the universal commutative $C^*$-algebra generated by the entries of a magic unitary matrix. See \cite{Wang} for details.\smallskip

The complete study of $S_4^+$ is in \cite{BaBi3}. We will study some subgroups of $A_s(n)$ using vertex-transitives graphs as in \cite{BaBi11}. 

We now recall the definition of the quantum automorphism group of a finite graph $X$ using \cite{BaHomog,BiNew}. \medskip

For a finite graph $X$ with $n$ vertices, it is convenient to also call $X$ the set of vertices of $X$. The complement graph of $X$ will be denoted by $X^c$.  If $i$ and $j$ are two vertices of $X$ we use the notation $i\sim_X j$ when they are connected and $i\not\sim_X j$ when they are not (or simply $i\sim j$ when no confusion can arise). \smallskip


\begin{defin} The \emph{adjacency matrix} of $X$ is the matrix $d_X=(d_{ij})_{1\le i,j\le n}\in M_n(0,1)$ given by $d_{ij}=1$ if $i,j$ are connected by an edge, and $d_{ij}=0$ if not. 

\end{defin}

The classical automorphism group of $X$ will be denoted by $\mathrm{Aut}(X)$ (this is a subgroup of $S_n$) and we have the following way to characterize its elements.

\begin{prop}\label{00} Identifying $\sigma\in S_n$ to the associated permutation matrix $P_{\sigma}\in M_n(\{0,1\})$, we have:
$$\sigma\in\mathrm{Aut}(X)\Longleftrightarrow d_XP_{\sigma}=P_{\sigma}d_X$$
\end{prop}\smallskip

This characterization in the classical case leads to the following natural definition of the quantum automorphism group of a finite graph, see \cite{BaHomog}.

\begin{defin} Associated to a finite graph $X$ is the quantum permutation algebra
$$A(X)=A_s(n)/\langle d_Xu=ud_X\rangle$$
where $n$ is the number of vertices of $X$.

\end{defin}\smallskip

The quantum automorphism group corresponding to $A(X)$ is the quantum automorphism group of $X$, denoted $\mathbb{G}_X$. In this way we have a commuting diagram of Woronowicz algebras:

\[\xymatrix{
     A_s(n)=C(S_n^+) \ar[rr]\ar[d]&  & A(X)= C(\mathbb{G}_X) \ar[d]\\
      C(S_n) \ar[rr]  & & C(\mathrm{Aut}(X))
    }\]
with the kernel of the right arrow being the commutator ideal of $A(X)$.  \smallskip

\begin{example} For the graph with $n$ vertices and no edges we have $A(X)=A_s(n)$, so $\mathbb{G}_X=S_n^+$. Moreover  we have $A(X^c)= A(X)$,
because $ud_X=d_Xu$ and $ud_{X^c}=d_{X^c}u$ are equivalent when $u$ is magic unitary.

If $X=C_n$ is the $n$-cycle graph one can show that for $n\ne 4$, $A(C_n)$ is commutative, thus $A(C_n)=C(\mathrm{Aut}(C_n))$ and therefore $\mathrm{Aut}(C_n)=\mathbb{G}_{C_n}=\mathcal{D}_n$,
where $\mathcal{D}_n$ is the $n$-dihedral group.
For more examples see \cite{BaBi11}.
\end{example}



\medskip

\section{Intertwinner spaces in the classical case}

We use the notations from  \cite{bs}. Let $G$ be a subgroup of $\mathbb{S}_n$ (we identify the elements $\sigma$ of $G$ to their associated permutation matrix $P_{\sigma}\in\M_n(\{0,1\})$). The intertwinner spaces are the following $C_G(k,l)$:
$$C_G(k,l)=\{T\in\mathrm{Hom}((\C^n)^{\otimes k},(\C^n)^{\otimes l}):=C(k,l)\mid\ \forall g\in G, Tg^{\otimes k}=g^{\otimes l}T\}$$
with $(k,l)\in\N\times\N$.

\begin{remark}By Frobenius duality we have
$$C_{G}(k,l)\simeq C_{G}(0,k+l)$$
\end{remark}

We denote by $C_G$ the set of all $C_G(k,l)$ and we define the elements $U, M, C$ as follow:
$$U\in C_G(0,1):\ \ U(1)=\sum_{k=0}^{n-1}{e_k}\ \ \ \ \ \ \ \ $$
$$M\in C_G(2,1):\ \ M(e_i\otimes e_j)=\delta_{i,j}e_i\ $$
$$S\in C_G(2,2):\ \ S(e_i\otimes e_j)=e_j\otimes e_i$$

and their duals:

$$U^*\in C_G(1,0):\ \ U^*(e_i)=1\ \ \ \ \ \ \ \ $$
$$M^*\in C_G(1,2):\ \ M^*(e_i)=e_i\otimes e_i $$
$$S^*\in C_G(2,2):\ \ S^*=S \ \ \ \ \ \ \ \ \ \ \ $$

From  \cite{bs} we get the following proposition on the structure of $C_G$.

\begin{prop}\label{p1}{[\cite{bs}, Proposition 1.2]} The collection of vector spaces $C_G(k,l)$ is a tensor category with duals, in the sense that it has the following properties:
\begin{enumerate}
\item $(T,T'\in C_G)\Rightarrow T+T'\in C_G$
\item $(T,T'\in C_G)\Rightarrow T\otimes T'\in C_G$
\item If $T,T'\in C_G$ are composable, then $T\circ T'\in C_G$
\item $(T\in C_G)\Rightarrow T^*\in C_G$
\item $I_d\in C_G(1,1)$
\end{enumerate}
\end{prop}

We use a version of a Tannaka-Krein duality theorem about Woronowicz algebra (in \cite{Woro3}). The use we make is an analogous of theorem $1.4$ in \cite{bs} for the study of compact groups containing $\mathbb{S}_n$.

\begin{theorem} The construction $G\to C_G$ induce a bijection between subgroups $G\subset\mathbb{S}_n$ and the symmetric tensor categories with duals, $\mathcal{C}$, satisfying
$$C_{\mathbb{S}_n}\subset\mathcal{C}\subset C_{\{1\}}$$
where $C_{\{1\}}(k,l)=C(k,l)$.
\end{theorem}

 $C_G$ is spanned as symmetric tensor category by endomorphism 
 that characterize the relations satisfing by the coefficients of element of $G$ (view as matrices). The matrix $u=(u_{ij})_{0\le i,j\le n-1}$ is in $\mathbb{S}_n$ if:
\begin{enumerate}[i.]
\item $\ds\sum_{k=0}^{n-1}{u_{ik}}=\ds\sum_{k=0}^{n-1}{u_{kj}}=1$
\item $u_{ij_1}u_{ij_2}=\delta_{j_1,j_2}u_{ij_1}$ et $u_{i_1j}u_{i_2j}=\delta_{i_1,i_2}u_{i_1j}$
\item The $u_{ij}$ commute
\end{enumerate}

Then we can check that:
$$ (i)\Leftrightarrow U, U^*\in C_G(0,1)$$
$$ (ii)\Leftrightarrow M, M^*\in C_G(2,1)$$
$$ (iii)\Leftrightarrow S\in C_G(2,2)$$

so we get:

$$C_{\mathbb{S}_n}=\langle U, M, S\rangle_{+, \circ, \otimes, *}$$

There is a more explicit way to describe  $C_{\mathbb{S}_n}$ by using the set $P(k,l)$ of all partitions of $[1,\dots,k+l]$, see \cite{bs}.

\begin{remark} 
 $(3)+(4)+(ii)\Rightarrow (5)$ since $I_d=M\circ M^*$.
\end{remark}

When $G=\Aut(X)$ we use this following description

\begin{prop}\label{TannaClassique} If $X$ is a vertex-transitive finite graph, then
$$C_{\Aut(X)}=\langle U, M, S, d_X \rangle_{+,\circ,\otimes,*}$$
\end{prop}

\begin{proof}
Elements of $\Aut(X)$ are those of $\mathbb{S}_n$ which commute with $d_X$, and by definition 
$$ (d_X\in C_{\Aut(X)}(1,1))\Leftrightarrow(\forall \sigma\in\Aut(X),\ P_{\sigma}d_X=d_XP_{\sigma})$$
which gives us the result. 
\end{proof}

\medskip

\section{Spaces $C_{\Aut(X)}(1,1)$ and $\mathcal{T}^X$}

In this part $G$ is a transitive subgroup of $\mathbb{S}_n$ and we will consider that  $\mathbb{S}_n$ is the set of permutations over $\{0,1,\dots, n-1\}$.

\begin{defin} For all $i\in[0,n-1]$, we denote by $G_i$ the stabilisator of $i$ in $G$. Then we denote by $\mathcal{O}_0^s$ the $G_0$-orbit over $[0,n-1]$ with $\mathcal{O}_0^0=\{0\}$.

For $i\in[1,n-1]$ we choose $\sigma_i\in G$ such that $\sigma_i(0)=i$ and we define:
$$ \mathcal{O}_i^s:=\sigma_i(\mathcal{O}_0^s)$$
\end{defin}

\begin{prop} \label{D1} The $\mathcal{O}_i^s$ has the following properties:
\begin{enumerate}[i.]
\item The  $\mathcal{O}_i^s$, for $0\le s\le r$, are the $G_i$-orbit over $[0,n-1]$.
\item The definition of $\mathcal{O}_i^s$ does not depend on the choice of $\sigma_i\in G$ such that $\sigma_i(0)=i$.
\item $\forall\sigma\in G$, $\forall i,j\in[0,n-1]$, $\forall s_1,s_2\in[0,r]$ we have
$$\sigma(\mathcal{O}_i^{s_1})=\mathcal{O}_{\sigma(i)}^{s_1}\ \ \ \text{and}\ \ \ \sigma(\mathcal{O}_i^{s_1}\cap \mathcal{O}_j^{s_2})=\mathcal{O}_{\sigma(i)}^{s_1}\cap\mathcal{O}_{\sigma(j)}^{s_2}$$
\end{enumerate}
\end{prop}

\begin{proof} $(i)$: Let $j_0\in \mathcal{O}_i^s=\sigma_i(\mathcal{O}_0^s)$, $j_0=\sigma_i(i_0)$ with $i_0\in\mathcal{O}_0^s$.
\begin{eqnarray*} 
j\in\mathcal{O}_i^s=\sigma_i(\mathcal{O}_0^s)&\Leftrightarrow& \exists j'\in\mathcal{O}_0^s, \ j=\sigma_i(j')\\
&\Leftrightarrow&  \exists\sigma'\in G_0,\ j=\sigma_i\sigma'(i_0) \\
&\Leftrightarrow&   \exists\sigma'\in G_0,\ j=\sigma_i\sigma'\sigma_i^{-1}(j_0) \\
&\Leftrightarrow&   \exists\sigma'\in \sigma_i G_0\sigma_i^{-1},\ j=\sigma'(j_0) \\
&\Leftrightarrow&   \exists\sigma'\in G_{\sigma_i(1)},\ j=\sigma'(j_0) \\
&\Leftrightarrow&   \exists\sigma'\in G_{i},\ j=\sigma'(j_0) \\
&\Leftrightarrow&   j \text{ is in the } G_i\text{-orbit of}\ j_0
\end{eqnarray*}
which show that the $ \mathcal{O}_i^s$ are the annonced orbits.

$(ii)$: Let $\sigma\in G$ such that $\sigma(0)=i$, then
$$k\in\sigma(\mathcal{O}_0^s)\Leftrightarrow \exists k_0\in\mathcal{O}_0^s, \ k=\sigma(k_0) \Leftrightarrow \exists k_0\in\mathcal{O}_0^s, \ k=\sigma_i(\sigma_i^{-1}\sigma(k_0)) $$

since $\sigma_i^{-1}\sigma(0)=0$, so $\sigma_i^{-1}\sigma\in G_0$. By  definition of $\mathcal{O}_0^s$ we obtain $\sigma_i^{-1}\sigma(k_0)\in\mathcal{O}_0^s$ and we get
$$k\in\sigma(\mathcal{O}_0^s)\Longrightarrow k\in\sigma_i(\mathcal{O}_0^s) $$
which mean
$$\sigma(\mathcal{O}_0^s)\subset\sigma_i(\mathcal{O}_0^s) $$
then using the cardinal we find $\sigma(\mathcal{O}_0^s)=\sigma_i(\mathcal{O}_0^s)$  and the definition of $\mathcal{O}_i^s$ does not depend of the choice on $\sigma_i$. 

$(iii)$: The first equality comes from this following calculation: 
$$\sigma(\mathcal{O}_i^{s_1})=\sigma\sigma_i(\mathcal{O}_0^{s_1})=\mathcal{O}_{\sigma\sigma_i(0)}^{s_1}=\mathcal{O}_{\sigma(i)}^{s_1}$$
 and the second from:
$$\sigma(\mathcal{O}_i^{s_1}\cap\mathcal{O}_{j}^{s_2})=\sigma(\mathcal{O}_i^{s_1})\cap\sigma(\mathcal{O}_{j}^{s_2})=\mathcal{O}_{\sigma(i)}^{s_1}\cap\mathcal{O}_{\sigma(j)}^{s_2}$$
\end{proof}

\begin{defin} For every $s\in[0,r]$ we define $T_s\in\Hom(\C^n,\C^n)$ by
$$T_s(e_j)=\sum_{i\in\mathcal{O}_j^{s}}{e_i}$$
and we denote
$$\mathcal{T}^G:=\Vect(T_s\mid 0\le s\le r)$$
By definition the family $(T_0,\dots,T_r)$ is free and it's a basis of $\mathcal{T}^G$.
\end{defin}

\begin{remark} This definition of $T_s$ is not canonical because there is a choice made in the order of orbit's labels. But the definition of $\mathcal{T}^G$ is canonical.
\end{remark}

We now present an other way to build this endomorphism $T_s$ to get new properties about.

\begin{prop}\label{alternative} For $s\in[0,r]$, we denote $\mathcal{O}^s$ the orbit of the diagonal action of $G$ on $[0,n-1]\times [0,n-1]$ which contained $\mathcal{O}_0^s\times\{0\}$. The $(\mathcal{O}^0,\dots,\mathcal{O}^r)$ are orbits of this action and for all $s\in[0,r]$ we have
$$[T_s]_{i,j}=\delta_{(i,j)\in\mathcal{O}^s}$$
\end{prop}

\begin{proof}We first prove that all the orbits are of the form: $\mathcal{O}^s$.\smallskip

 Indeed if $(i,j)\in[0,n-1]\times[0,n-1]$ then there exists  $s\in[0,r]$ such that $i\in\mathcal{O}_j^s$. 

Let $\sigma\in G$ such that $\sigma(0)=j$. We obtain
$i\in\sigma(\mathcal{O}_0^s)$, so
$$(i,j)\in\sigma(\mathcal{O}_0^s\times\{0\})$$
which mean $(i,j)\in\mathcal{O}^s$. It gives also
$$(i,j)\in\mathcal{O}^s\Longleftrightarrow i\in\mathcal{O}_j^s$$

And by definition we have 
$$[T_s]_{i,j}=\delta_{i\in\mathcal{O}_j^s}$$

 as required.
\end{proof}

This description exactly correspond to a coherent configuration as in \cite{David}.

\begin{lem}\label{l1}If $\sigma\in G\subset\mathbb{S}_n$ then,
$$P_{\sigma}T_s=T_sP_{\sigma}$$
\end{lem}

\begin{proof} We recall that by definition $P_{\sigma}(e_i)=e_{\sigma(i)}$. 

Let $s\in[0,r]$ and $j\in[0,n-1]$:
$$P_{\sigma}T_s(e_j)=\sum_{i\in\mathcal{O}_j^s}{e_{\sigma(i)}}=\sum_{k\in\sigma(\mathcal{O}_j^s)}{e_{k}}=\sum_{k\in\mathcal{O}_{\sigma(j)}^s}{e_{k}}=T_s(e_{\sigma(j)})=T_sP_{\sigma}(e_j)$$

so we get $P_{\sigma}T_s=T_sP_{\sigma}.$
\end{proof}

\begin{lem}\label{l2}For $f\in\Hom(\C^n,\C^n)$ the following propositions are equivalent:
\begin{enumerate}
\item $\forall\sigma\in G,\ P_{\sigma}f=fP_{\sigma}$
\item  $\forall\sigma\in G,\forall i,j\in [0,n-1],\ \  a_{i,j}=a_{\sigma(i),\sigma(j)}$ where the $a_{i,j}$ are coefficients of the matrix of $f$ in the canonical basis $(e_0,\dots, e_{n-1})$.
\item $f\in\mathcal{T}^G$
\end{enumerate}
\end{lem}


\begin{proof} $(1)\Leftrightarrow (2)$:
\begin{eqnarray*} 
\forall\sigma\in G,\ P_{\sigma}f=fP_{\sigma}&\Leftrightarrow& \forall\sigma\in G,\forall i,j\in [0,n-1],\ \ [P_{\sigma}f]_{i,j}=[fP_{\sigma}]_{i,j}\\
&\Leftrightarrow& \forall\sigma\in G,\forall i,j\in [0,n-1],\ \ a_{\sigma^{-1}(i),j}=a_{i,\sigma(j)  } \\
&\Leftrightarrow&   \forall\sigma\in G,\forall i,j\in [0,n-1],\ \  a_{i,j}=a_{\sigma(i),\sigma(j)}
\end{eqnarray*}

$(3)\Rightarrow (1)$: Come quickly from lemma \ref{l1}

$(2)\Rightarrow (3)$: Let $s\in[0,r]$ and $j_1\ne j_2\in \mathcal{O}_0^s$. 

Then there exists $\sigma\in G_0$ such that $\sigma(j_1)=j_2$. And by assumptions we have
$$  a_{j_1,0}=a_{\sigma(j_1),\sigma(0)}=a_{j_2,0}:=a_{s,0}$$
and then
$$f(e_0)=\sum_{s=0}^r{a_{s,0}T_s(e_0)}$$

Let $i\in[1,n-1]$. As $G$ act transitively it exists $\sigma\in G$ such that $\sigma(0)=i$. Hence
$$fP_{\sigma}(e_0)=f(e_i)$$
$$\text{so}\ \ P_{\sigma}f(e_0)=\sum_{s=0}^r{a_{s,0}P_{\sigma}T_s(e_0)}
=\sum_{s=0}^r{a_{s,0}T_s(e_{i})}$$
 by lemma \ref{l1}. Which mean
$$f(e_i)=\sum_{s=0}^r{a_{s,0}T_s(e_{i})}$$
hence $$f=\sum_{s=0}^r{a_{s,0}T_s}\in\mathcal{T}^G$$
\end{proof}

\begin{remark} For  $(2)\Leftrightarrow (3)$ we also can see that $(2)$ mean that the function $(i,j)\mapsto a_{i,j}$ is constant over the orbits $\mathcal{O}^s$, which mean $f\in\mathcal{T}^G$ by using proposition \ref{alternative}.
\end{remark}

From now $X$ will denote a finite vertex-transitive graph. The set of vertices of $X$ will be denoted by $[0,n-1]$ and its adjacency matrix by $d_X$ in the canonical basis $(e_0,\dots,e_{n-1})$ of $\C^n$. Then, $G=\Aut(X)$ satisfy all the assumptions to define
 $$\Aut_i(X)=G_i \ \ \text{ and } \ \ \TX:=\mathcal{T}^{\Aut(X)}$$

\begin{remark} Lemma \ref{l2} show that $\TX$ is stable by composition and that $d_X\in\TX$ so we have: 
$$ P_{\sigma}d_X=d_XP_{\sigma}\Leftrightarrow\sigma\in\Aut(X)\Leftrightarrow \forall s\in[0,r],\  P_{\sigma}T_s=T_sP_{\sigma}$$

Moreover the definition of $\TX$ only depend on $\Aut(X)$ so
$$\TX=\mathcal{T}^{X^c}$$ since $\Aut(X)=\Aut(X^c)$.
\end{remark}



To resume, for this automorphism group of a vertex-transitive graph, we have this proposition.

\begin{prop}\label{TannaClassique} If $X$ is a vertex-transitive finite graph, then
$$C_{\Aut(X)}=\langle U, M, S, d_X \rangle_{+,\circ,\otimes,*}$$
$$ C_{\Aut(X)}(1,1)=\TX$$
\end{prop}

\begin{proof}
The first equality is a recall from proposition \ref{TannaClassique} and the second is a direct consequence of Lemma \ref{l2} which say that $ C_{G}(1,1)=\TG$ and for $G=\Aut(X)$ we have $\TG=\TX$.
\end{proof}

In the following pages we want to study an equivalent of this intertwinner spaces for the quantum case.


\medskip

\section{The quantum case.}\label{Partdual}

\begin{defin}A $n$-quantum permutation algebra is a quotients $\mathcal{A}$ of 
$$A_s(n)=\langle u_{ij}, 1\le i,j\le n\mid u \text{ unitary matrix } \rangle$$
They correspond to quantum subgroups of $S_n+$.
\end{defin}

Always using the same notation of \cite{bs}, the intertwinner spaces are:

$$C_{\mathcal{A}}(k,l)=\{T\in C(k,l)\mid Tu^{\otimes k}=u^{\otimes l}T\}$$
where $u^{\otimes k}$ is the $n^k\times n^k$-matrix $(u_{i_1j_1}\dots u_{i_kj_k})_{i_1\dots i_kj_1\dots j_k}$.

\begin{remark}By Frobenius duality we still have:
$$C_{\mathcal{A}}(k,l)\simeq C_{\mathcal{A}}(0,k+l)$$
\end{remark}

The collection of the $C_{\mathcal{A}}(k,l)$ for $(k,l)\in\N\times\N$, denoted $C_{\mathcal{A}}$, is still a tensor category with duals in the sense of proposition \ref{p1}. The Tannaka-Krein duality is as follows.

\begin{theorem} Construction of $\mathcal{A}\to C_{\mathcal{A}}$, induce a bijection between the $n$-permutations quantum algebra and the tensor categories with duals, $\mathcal{C}$, satisfying
$$C_{A_s(n)}\subset\mathcal{C}\subset C_{\{1\}}$$
\end{theorem}

Similary to the classical case we have:
$$C_{A_s(n)}=\langle U, M \rangle_{+,\circ,\otimes,*}$$
 
This result appears first in \cite{Sygeco} and more recently in \cite{bc}. In \cite{bs}, we can fine a another classical description of $C_{A_s(n)}(k,l)$, using non crossing partition of $[0,n-1]$.

 First, we can check that,
$$C_{A_s(n)}(0,1)=\Vect(e),\ \ \text{where }e=\sum_{i=0}^{n-1}{e_i}$$
$$C_{A_s(n)}(1,1)=\Vect(\Id,J)$$
where $J$ is the matrix whose all entries are equal to one.
 
\begin{lem}\label{TG112} Let $\mathcal{A}$ be an $n$-permutation quantum algebra and $G^0\subset\mathbb{S}_n$ the permutation group such that:
$$com(\mathcal{A})=C({G^0})$$
where $com(\mathcal{A})$ is the abelianised algebra of $\mathcal{A}$.
Then
$$C_{\mathcal{A}}\subset C_{G^0}$$
and in particular $$C_{\mathcal{A}}(1,1)\subset \mathcal{T}^{G_0}$$
\end{lem}

\begin{proof} This result is a direct consequence of the definition. The abelianisation correspond to the add of $S$ in the generators of $C_{\mathcal{A}}$.
\end{proof}

For $A(X)$ we have $com(A(X))=C(\Aut(X))$ so $G^0=\Aut(X)$ and this lemma tell us that
$$C_{A(X)}(1,1)\subset\TX$$

We obtain the natural generalisation of the proposition \ref{TannaClassique} for $A(X)$ in the quantum case.

\begin{prop}\label{TannaQ}For every finite graph $X$ we have:
$$C_{A(X)}=\langle U, M, d_X \rangle_{+,\circ,\otimes,*}$$

\end{prop}

\begin{proof}  $A(X)$ is defined by the magic unitary matrix characterized by the following relations :
\begin{enumerate}[i.]
\item $\ds\sum_{k=0}^{n-1}{u_{ik}}=\ds\sum_{k=0}^{n-1}{u_{kj}}=1$
\item $u_{ij_1}u_{ij_2}=\delta_{j_1,j_2}u_{ij_1}$ et $u_{i_1j}u_{i_2j}=\delta_{i_1,i_2}u_{i_1j}$
\item $ud_X=d_Xu$
\end{enumerate}
and as
$$ (i)\Leftrightarrow U\in C_{A(X)}(0,1)\ \ \ \ \  (ii)\Leftrightarrow M\in C_{A(X)}(2,1)\ \ \ \ \  (iii)\Leftrightarrow d_X\in C_{A(X)}(1,1)$$
it ends the proof.
\end{proof}

\begin{cor}\label{CorS} $X$ has no quantum symmetry if and only if $S\in C_{A(X)}$.
\end{cor}

\begin{proof} We know that $X$ has no quantum symmetry if and only if  $C_{A(X)}=C_{\Aut(X)}$ but $C_{A(X)}=\langle U, M, d_X \rangle_{+,\circ,\otimes,*}$ and $C_{\Aut(X)}=\langle U, M, S, d_X \rangle_{+,\circ,\otimes,*}$ so it works.

We can also check that $S\in C_{A(X)}$ mean that the $u_{i,j}$ (who spanned $A(X)$) commute hence by definition $X$ has no quantum symmetry if and only if $A(X)$ is commutative.
\end{proof}

Now we want to understand better $C_{A(X)}(1,1)$ using the space $\TX$ as in the classical case. For that, we use the following Hadamard product on $\Hom(\C^n,\C^n)$.

\begin{defin} Let $f,g\in \Hom(\C^n,\C^n)$, we denote $f_\times g\in\Hom(\C^n,\C^n)$ the Hadamard product  of $f$ and $g$ define by
$$f_\times g=M\circ(f\otimes g)\circ M^*$$
\end{defin}

\begin{prop}\label{PropProduit}  This product satisfy the following properties:
\begin{enumerate}
\item If $f,g\in C_{A(X)}(1,1)$ then $f_\times g\in C_{A(X)}(1,1)$.
\item Let $f,g\in \Hom(\C^n,\C^n)$, written on the matrix-form $f=(a_{ij})$ and $g=(b_{ij})$. Then we have $f_{\times} g=(a_{ij}b_{ij})$.
\item If $f,g\in\TX$ are written on the form $$f=\sum_{i=0}^r{a_i^fT_i},\ \ \ \ g=\sum_{i=0}^r{a_i^gT_i}$$
then we have
$$f_\times g=\sum_{i=0}^r{a_i^fa_i^gT_i}\in\TX$$
which mean $a_i^{f_\times g}=a_i^fa_i^g$.
\end{enumerate}
\end{prop}

\begin{proof} \

$(1)$: Trivial since $M,M^*\in C_{A(X)}(1,1)$ and $C_{A(X)}$ is stable by composition and tensor product.

$(2)$:
\begin{eqnarray*} 
f_{\times}g(e_j)=M(f(e_j)\otimes g(e_j))&=&  M\left(    \sum_{i_1=0}^{n-1}{a_{i_1,j}e_{i_1}}\otimes  \sum_{i_2=0}^{n-1}{b_{i_2j}e_{i_2}}\right)\\
&=&M\left(    \sum_{i_1,i_2=0}^{n-1}{a_{i_1,j}b_{i_2,j}(e_{i_1}\otimes e_{i_2})}\right)\\
&=& \sum_{i_1,i_2=0}^{n-1}{a_{i_1,j}b_{i_2,j}M(e_{i_1}\otimes e_{i_2})}\\
&=&\sum_{i=0}^{n-1}{a_{ij}b_{ij}e_i}\\
\end{eqnarray*}

which gives us the result.

$(3)$: When $f,g\in\TX$ we have this following calculation:
\begin{eqnarray*} 
f_{\times}g(e_j)&=&  M\left(    \sum_{i_1=0}^r{\alpha_{i_1}^fT_{i_1}(e_j)}\otimes  \sum_{i_2=0}^r{\alpha_{i_2}^gT_{i_2}(e_j)}\right)\\
&=&  \sum_{i_1,i_2=0}^r{  \alpha_{i_1}^f\alpha_{i_2}^g\left(   M(T_{i_1}(e_j) \otimes T_{i_2}(e_j)   \right) }  \\
&=&  \sum_{i=0}^q{  \alpha_{i}^f\alpha_i^gT_{i}(e_j)  }  \ \ \ \ \ \text{since}\ \ \mathcal{O}_{i_1}^j\cap \mathcal{O}_{i_2}^j=\delta_{i_1,i_2}\mathcal{O}_{i_1}^j\\
\end{eqnarray*}
So $f_\times g=\ds\sum_{i=0}^r{  \alpha_{i}^f\alpha_i^gT_{i}}$.\smallskip\smallskip
\end{proof}

\noindent This theorem make summarize the link between $\Aut(X)$, $A(X)$ and $\TX$.

\begin{theorem}\label{T1} Let $X$ be a vertex-transitive graph, then
$$C_{\Aut(X)}=\langle U, M, S, d_X\rangle_{+,\circ,\otimes,*}=\langle U, M, S, \mathcal{T}^X\rangle_{+,\circ,\otimes,*}:=\langle U, M, S, T_1,\dots, T_r\rangle_{+,\circ,\otimes,*}$$
$$C_{A(X)}=\langle U, M, d_X\rangle_{+,\circ,\otimes,*}=\langle U, M, \C[d_X]\rangle_{+,\circ,\otimes,*}$$

\end{theorem}

\begin{proof}
With proposition \ref{TannaClassique} we obtain
$$\langle U, M, S, \mathcal{T}^X\rangle_{+,\circ,\otimes,*}\subset C_{\Aut(X)}=\langle U, M, S, d_X\rangle_{+,\circ,\otimes,*}$$ using $d_X\in\TX$ we obtain  the inverse inclusion and the equality.

The second equality is the proposition \ref{TannaQ}.
\end{proof}

\begin{defin} A vertex-transitive graph $X$ is called \textbf{$\mathcal{B}$-clos} if 
$$C_{A(X)}(1,1)=\TX$$
which is equivalent to  $\TX\subset C_{A(X)}$.
\end{defin}

The quantum symmetries of a graph $\mathcal{B}$-clos are easier to study because he satisfies: 
$$C_{A(X)}=\langle U, M, \TX \rangle_{+,\circ,\otimes,*}$$

which mean $A(X)$ only depend on its automorphism group, $\Aut(X)$, in the sense of this following theorem.

\begin{theorem}\label{T2} Let $X$ and $Y$  be two graphs $\mathcal{B}$-clos with the same number $n$ of vertices. Then the following three propositions are equivalent:
\begin{enumerate}
\item $\mathcal{T}^X=\mathcal{T}^Y$
\item $\Aut(X)=\Aut(Y)$
\item $A(X)=A(Y)$
\end{enumerate}
\end{theorem}

\begin{proof}\indent

\underline{$(1)\Rightarrow (3)$}: 
\begin{eqnarray*} 
\mathcal{T}^X=\mathcal{T}^Y&\Rightarrow&\langle U, M, \mathcal{T}^X\rangle_{+,\circ,\otimes,*}=\langle U, M, \mathcal{T}^Y\rangle_{+,\circ,\otimes,*}  \\
&\Leftrightarrow& C_{A(X)}=C_{A(Y)} \\
&\Leftrightarrow& A(X)=A(Y) \\
\end{eqnarray*}

\underline{$(3)\Rightarrow (2)$}: By passing to the abelianised.

\underline{$(2)\Rightarrow (1)$}: Since $\mathcal{T}^X$ is canonicaly define by $\Aut(X)$.
\end{proof}

\medskip

\section{Applications to circulant $p$-graphs}

In this section we study quantum symmetries of some circulant graphs, define as follow.

\begin{defin}A graph with $n$ vertices is called circulant if its automorphism group contains a cycle of length $n$, and hence a copy of the cycle $\Z_n$. We also call $n$-circulant graph for a circulant graph with $n$ vertices.
\end{defin}

If we denoted by $(i_1,i_2,\dots,i_n)$ a cycle of length $n$ in a $n$-circulant graph then all pairs $(i_k,i_{k+1})$ for all $k\in[1,n]$ (with $i_{n+1}:=i_1$) have the same nature (connected or not). So, modulo complementation, we just need to study quantum symmetries where this above pairs are connected. 

As Banica and Bichon in \cite{BaBi11} we will concentrate our study on vertex-transitive graphs. First we remark that for any prime $p$  all vertex-transitive graph with $p$ vertices are circulant. Then vertex-transitive circulant graphs are exactly those describe by this following definition.

\begin{defin}\label{D} We denoted by $C_n(k_1,\dots,k_r)$, with $1<k_1<\dots<k_r\le\lfloor \frac{n}{2}\rfloor$, $k_i\in\N$, the graph obtained by drawing the $n$-cycle $C_n$, then connecting all pairs of vertices at distance $k_i$ on the circle, for all $i$.

 Which mean that for all $i,j\in[0,n-1]$:
$$i\sim j \Longleftrightarrow (i-j)\mod \ n\ \in\{\pm1,\pm k_1,\dots,\pm k_r\}$$
\end{defin}

\begin{example} The complete graph with $n$ vertices, name $K_n$ is circulant and here are two other examples of circulant graphs:
\begin{enumerate}[i.]
\item $X=C_6(2)=(3K_2)^c$:
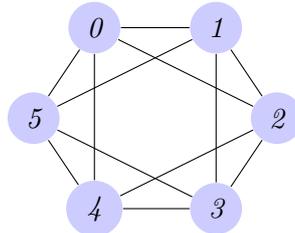
$$\begin{tikzpicture}
  [scale=.8,auto=left,every node/.style={circle,fill=blue!20}]
  \node (n1) at (4,5.5) {0};
  \node (n2) at (6,5.5)  {1};
  \node (n3) at (7,4)  {2};  
  \node (n4) at (6,2.5)  {3};
  \node (n5) at (4,2.5)  {4};
  \node (n6) at (3,4) {5};

  \foreach \from/\to in {n1/n2,n2/n3,n3/n4,n4/n5,n5/n6,n6/n1,n1/n3,n2/n4,n3/n5,n4/n6,n6/n2,n1/n5}
    \draw (\from) -- (\to);
\end{tikzpicture}$$

\item $X=C_{13}(3,4):$
$$\begin{tikzpicture}
  [scale=.8,auto=left,every node/.style={circle,fill=blue!20}]
  \node (n1) at (0,10.4) {0};
  \node (n2) at (2.2,9.6)  {1};
  \node (n3) at (4.4,8)  {2};  
  \node (n4) at (5.2,5.6)  {3};
  \node (n5) at (4.8,3.2)  {4};
  \node (n6) at (3.6,1.2) {5};
  \node (n7) at (1.4,0)  {6};
  \node (n8) at (-1.4,0)  {7};
  \node (n9) at (-3.6,1.2) {8};
   \node (n10) at (-4.8,3.2) {9};
  \node (n11) at (-5.2,5.6)  {10};
  \node (n12) at (-4.4,8)  {11};
  \node (n13) at (-2.2,9.6) {12};
 
  \foreach \from/\to in {n1/n2,n2/n3,n3/n4,n4/n5,n5/n6,n6/n7,n7/n8,n8/n9,n9/n10,n10/n11,
n11/n12,n12/n13,n13/n1,
n1/n4,n1/n5,n2/n5,n2/n6,n3/n6,n3/n7,n4/n7,n4/n8,
n5/n8,n5/n9,n6/n9,n6/n10,n7/n10,n7/n11,n8/n11,n8/n12,
n9/n12,n9/n13,n10/n13,n10/n1,n11/n1,n11/n2,n12/n2,n12/n3,n13/n3,n13/n4}
    \draw (\from) -- (\to);
\end{tikzpicture}$$
\end{enumerate}
\end{example}

We also need more definition (from \cite{BaBi2}) to study those circulant graphs.

For a $n$-circulant graph we suppose that their vertices are elements of $\Z_n$. Then for every $k,i,j\in\Z_n$:$\ \ i\sim j\Rightarrow i+k\sim j+k$. We denote by $\Z_n^*$ the group of invertible elements of the ring $\Z_n$. The following definition comes from \cite{BaBi2}.

\begin{defin}Let $X$ be a circulant graph with $n$ vertices:
\begin{itemize}
\renewcommand{\labelitemi}{$\bullet$}
\item The set $S\subset\Z_n$ is given by $i\sim j\Longleftrightarrow j-i\in S$.
\item The group $E\subset\Z_n^*$ consists of elements $a\in\Z_n^*$ such that $aS=S$.
\item The order of $E$ is denoted $k$, and called type of $X$.
\end{itemize}
\end{defin}

We can check that the set $S$ is $\{\pm1,\pm k_1,\dots,\pm k_r\}$ from definition \ref{D}. In the general case the type of a circulant graph is always even because $-1$ of order $2$ is in $E$.

\begin{remark}\label{Lienpk}By definition $E$ is a subgroup of $\Z_n^*$ so when $n=p$ is prime $p-1$ is divided by the order of $E$ so there exists $\lambda$ in $\N^*$ such that:
$$p=\lambda k+1$$
\end{remark}

The study in \cite{BaBi2} use the the notion of $2$-maximal as follows.

\begin{defin}{[\cite{BaBi2}, Definition 4.2]} Let $(R,+,.)$ a ring such that  $2\in R^*$ where $R^*$ is the group of invertible element of $R$. An even subgroup $G\subset R^*$  is called $2$-maximal if:
$$a-b=2(c-d)$$
with $a,b,c,d\in G$ implies $a=\pm b$.
\end{defin}

\begin{theorem}\label{SQ}{[\cite{BaBi2}, Theorem 4.4]} If $E\subset\Z_p$ is $2$-maximal (with $p\ge 5$) then $X$ has no quantum symmetry.
\end{theorem}

\begin{theorem}\label{Borne0}{[\cite{BaBi2}, Theorem 5.3]} A type $k$ circulant $p$-graphe with $p> 6^{\varphi(k)}$ has no quantum symmetry. Where $\varphi$ is the Euler function. 
\end{theorem}

It means that we just need to study circulant $p$-graph which satisfy $p\leqslant  6^{\varphi(k)}$. 

For $k\leqslant 12$ we have:
\begin{center}
\begin{tabular}{ | c | c  | c | c | c | c | }
 \hline                 
   $k$ & $4$ & $6$ & $8$ & $10$ & $12$ \tabularnewline \hline
$6^{\varphi(k)}$ & $36$ & $36$ & $1296$ & $1296$ & $1296$ \\ \hline 
 \end{tabular}
\end{center}

To study all this circulant graphs with $p\leqslant  6^{\varphi(k)}$ we  look at the space $\TX$ for this graph. For the orbitals of this $p$-graphs we have those following results from \cite{JoyMorris}.

\begin{prop}\label{prepa1} If $X$ is a non trivial circulant $p$-graph then for all $i\in[0,p-1]$, we have the following isomorphism
\begin{eqnarray*}
\Phi_i: E&\to &\Aut_i(X)\\
e&\mapsto & (\sigma_e: i+k\mapsto i+ke) \\
\end{eqnarray*}
\end{prop}

It means that for all $s\in[1,r]$, the orbitals $\mathcal{O}_0^s$ are of the form:
$$\mathcal{O}_0^s=y_sE$$
with a choice of $y_s$ such that:
$$1=y_1<y_2<\dots <y_r$$
Then we also obtain:
$$\mathcal{O}_i^s=i+y_sE$$

\begin{prop}\label{pcoh}Every circulant $p$-graphs satisfy
$$\C[d_X]=\TX$$
so every circulant $p$-graph is $\mathcal{B}$-clos.
\end{prop}

\begin{proof}From proposition \ref{TannaClassique} we already know that $\C[d_X]\subset \TX$. We prove the other inclusion by checking the dimensions of this two spaces. First, the dimension of $\TX$ is the number "$r+1$" of $\mathrm{Aut}_0(X)$-orbit and we know from \ref{prepa1} that this orbital are $\{0\}$ and the $y_sE$ with $s\in[1,r]$. So, except the trivial one, this orbital are all of length $k=|E|$ so we have: $1+r\times k=p$, which gives us: $$\dim\left(\TX\right)=r+1=\dfrac{p-1}{k}+1$$
Moreover, the dimension of $\C[d_X]$ is the number of eigen values of $d_X$ but, from \cite{BaBi2}, all the non trivial eigenspaces of $d_X$ are $k$-dimensionnal so we have: 
$$\dim\left(\C[d_X]\right)=\dfrac{p-1}{k}+1=\dim\left(\TX\right)$$
which ends the proof.
\end{proof}

This result with theorem \ref{T2} gives us this following proposition.

\begin{prop} Let $X_1$ and $X_2$ be two circulant $p$-graphs, then
$$E_1=E_2\Longrightarrow A(X_1)=A(X_2)$$
\end{prop}

\begin{proof}From \cite{Al} or \cite{JoyMorris} we know that $E$ characterizes $\mathrm{Aut}(X)$ for $p$-graphs. So with theorem \ref{T2} we have:
$$E_1=E_2\Longrightarrow \mathcal{T}^{X_1}=\mathcal{T}^{X_1}\Longrightarrow A(X_1)=A(X_2)$$
\end{proof}

This proposition is usefull because it says that the study of quantum symmetries for $p$-graphs depend only on $E$. Then we can reduce our study to the case $S=E$ since circulant-graph satisfy $E\subset S$. For $p$ and $k$ fixed there exists at most one graph to study since $E$ is the unique subgroup of order $k$ in $(\Z_p^*,\times)$.

\begin{example} There exist $15$ graphs of type $6$ for $p=31$ which are all of the form $C_{31}(5,6,U)$, where $U$ is any union of $(2,10,12)$, $(4,7,11)$, $(3,13,15)$ and $(8,9,14)$ (except union of the fourth). In all the cases, $E=\{\pm1,\pm5,\pm6\}$ so we just need to study $C_{31}(5,6)$.
\end{example}

We now give a new way to find if $X$ has quantum symmetries, using the duality of the corollary \ref{CorS}. We recall that the number of non trivials orbits of  $[0,p-1]$ under the action of $\Aut_0(X)$ is $r$ with
$$p=1+r k$$
and $\mathcal{O}_i^s=i+y_sE$ (with $y_1=1$). We can see that $r=\lambda$ from remark \ref{Lienpk}. 

\begin{prop}\label{NoSym} Let $X$ be a circulant $p$-graph such that
$$\text{for every}\  s \ \text{in}\ [1,r],\ \ \text{there exist }\  t_s^1,t_s^2\\ \text{in}\ [1,r] \text{ such that}\ |\mathcal{O}_0^{t_s^1}\cap\mathcal{O}_{y_s}^{t_s^2}|=1$$
Then $X$ has no quantum symmetry.
\end{prop}

To prove it we need the following lemma, characteristic of circulant $p$-graphs.

\begin{lem}\label{prepa2} Let $i,j,k\in[0,p-1]$ and $s,s_i, s_j\in[1,r]$ such that
$$j\in\mathcal{O}_i^s \ \text{ and }\ \mathcal{O}_i^{s_i}\cap\mathcal{O}_j^{s_j}=\{k\},$$
then
$$\mathcal{O}_i^{s}\cap\mathcal{O}_k^{s_j}=\{j\}\ \ \ \ \text{ and }\ \ \ \ \mathcal{O}_j^{s}\cap\mathcal{O}_k^{s_i}=\{i\}$$
\end{lem}

\begin{proof} We know by symmetry of element in $\mathcal{T}^X$ that $j\in\mathcal{O}_i^s\Leftrightarrow i\in\mathcal{O}_j^s$ hence, by symmetry, we just need to prove $\mathcal{O}_i^{s}\cap\mathcal{O}_k^{s_j}=\{j\}$.\smallskip\smallskip

$k\in\mathcal{O}_j^{s_j}\Rightarrow j\in\mathcal{O}_k^{s_j}$ so we have $j\in\mathcal{O}_i^{s}\cap\mathcal{O}_k^{s_j}$.\smallskip\smallskip

We can now prove that $\mathcal{O}_i^{s}\cap\mathcal{O}_k^{s_j}$ has no other element. \smallskip

Assume $j, j'\in \mathcal{O}_i^{s}\cap\mathcal{O}_k^{s_j}$. We have
$$k\in\mathcal{O}_i^{s_i}\cap\mathcal{O}_j^{s_j}\cap\mathcal{O}_{j'}^{s_j}$$
Moreover as $j,j'\in\mathcal{O}_i^{s}$ there exist $e,e'\in E$ such that
$$j=i+y_se\ \text{ and }\ j'=i+y_se'$$

We denote $\sigma:=\Phi_i(e'^{-1}e)$ where $\Phi_i$ is the isomorphism of the proposition \ref{prepa1}. Hence we have  $\sigma(i)=i$ and $\sigma(j')=i+y_se'e'^{-1}e=j$. Using $(iii)$ of the proposition \ref{D1} we obtain
$$\{\sigma(k)\}\subset\sigma(\mathcal{O}_i^{s_i}\cap
\mathcal{O}_{j'}^{s_j}\cap\mathcal{O}_{j}^{s_j})\subset \sigma(\mathcal{O}_i^{s_i}\cap
\mathcal{O}_{j'}^{s_j}) = \mathcal{O}_{\sigma(i)}^{s_i}\cap\mathcal{O}_{\sigma(j')}^{s_j}=
\mathcal{O}_{i}^{s_i}\cap\mathcal{O}_{j}^{s_j}=\{k\}$$

which mean $k=\sigma(k)$.  $k\in\mathcal{O}_{i}^{s_i}$ hence there exists  $e''\in E$ such that $k=i+y_{s_i}e''$ and we obtain 
$$i+y_{s_i}e''e'^{-1}e=\sigma(k)=k=i+y_{s_i}e''$$
and as $s_i\ne 0$ it implies that $e=e'$ hence $j=j'$, which ends the proof.
\end{proof}

Here is now the proof of the proposition \ref{NoSym}.

\begin{proof} We put $t_0^1=t_0^2=0$. We define the following element of  $C_{A(X)}$ for all $s\in[0,r]$.
$$h_s:=\left(M\otimes M\right) \circ \left(\id\otimes T_{t_s^2}\otimes T_{t_s^1}\otimes \id\right)\in C_{A(X)}(4,2)\  \ \ \ \ \ \ \ \ \ \ \ \ \ \ \ \ \ \ \ \ \ \ \ \ \ \ \ \ \ \ \ \ \ \ \ \ \ \ \ \ \ \ $$
$$f_{s}:=h_s\circ \left(\id\otimes (M^*\circ M)\otimes\id\right)\circ \left(T_s\otimes T_{t_s^1}\otimes T_{t_s^2}\otimes T_s\right)  \circ \left(M^*\otimes M^*\right) \in C_{A(X)}\left( 2,2   \right) $$
$$g_s:=(\id\otimes M)\circ(\id\otimes T_s\otimes\id)\circ(M^*\otimes \id)\in C_{A(X)}(2,2)\ \ \ \ \ \ \ \ \ \ \ \ \ \ \ \ \ \ \ \ \ \ \ \ \ \ \ \ \ \ \ \ \ \ \ \ \ \ \ \ \ \ \ \ \ \ \ \ \ \ \ \ \ \ \ \ \ \ \ \ \ \ \ \ \ \ \ \ \ \ \ \ \ \ \ \ \ \ \ \ \ \ \ \ $$
$$F:=\sum_{s=0}^{\lambda}{f_s\circ g_s}\in C_{A(X)}\left( 2,2   \right)  \ \ \ \ \ \ \ \ \ \ \ \ \ \ \ \ \ \ \ \ \ \ \ \ \ \ \ \ \ \ \ \ \ \ \ \ \ \ \ \ \ \ \ \ \ \ \ \ \ \ \ \ \ \ \ \ \ \ \ \ \ \ \ \ \ \ \ \ \ \ \ \ \ \ \ \ $$

First remark that $f_0=\id_{(\C^p)^{\otimes 2}}$.

Then for $s\in[1,r]$, we denote $k_s$ the element of $[0,p-1]$ such that
$$\mathcal{O}_{0}^{t_s^1}\cap\mathcal{O}_{y_s}^{t_s^2}=\{k_s\}$$

We now consider two elements $i$ and $j$ such that $j\in\mathcal{O}_{i}^{s}$. Let $\sigma\in\Aut(X)$ such that $\sigma(0)=i$. We know that $j_s\in\mathcal{O}_{0}^{s}$ hence $\sigma(j_s)\in\mathcal{O}_{i}^{s}$. So there exists $\sigma'\in\Aut_i(X)$ such that $\sigma'(\sigma(y_s))=j$. Then using $(iii)$ of proposition \ref{D1} we have
$$\{\sigma'\sigma(k_s)\}=\sigma'\sigma(\mathcal{O}_{0}^{t_s^1}\cap\mathcal{O}_{y_s}^{t_s^2})=
\mathcal{O}_{\sigma'\sigma(0)}^{t_s^1}\cap\mathcal{O}_{\sigma'\sigma(y_s)}^{t_s^2}=
\mathcal{O}_{i}^{t_s^1}\cap\mathcal{O}_{j}^{t_s^2}\ \ \ \ \ \ \ \ \ \ \ \ \ (1.1)$$

Then with lemma \ref{prepa2} we obtain also
$$\{i\}=\mathcal{O}_{\sigma'\sigma(k_s)}^{t_s^1}\cap\mathcal{O}_{j}^{s}\ \ \ \ \ \ \ \ \ \ \ \ \ (1.2)$$
$$\{j\}=\mathcal{O}_{i}^{s}\cap\mathcal{O}_{\sigma'\sigma(k_s)}^{t_s^2}\ \ \ \ \ \ \ \ \ \ \ \ \ (1.3)$$

Which allow us to make this following calculation for all $i,j\in[0,p-1]$ such that $j\in\mathcal{O}_{i}^{s}$:
\begin{eqnarray*}
f_s(e_i\otimes e_j)&=& h_s\circ\left(\id\otimes (M^*\circ M)\otimes\id\right)\left(  T_s(e_i)\otimes T_{t_s^1}(e_i)\otimes T_{t_s^2}(e_j)\otimes T_s(e_j) \right) \\
&=& h_s\left( T_s(e_i)\otimes   \sum_{k\in\mathcal{O}_{i}^{t_s^1}\cap\mathcal{O}_{j}^{t_s^2}}{(e_k\otimes e_k)}              \otimes T_s(e_j)            \right)  \\ 
&=&h_s \left(  T_s(e_i)\otimes   e_{\sigma'\sigma(k_s)}\otimes e_{\sigma'\sigma(k_s)}             \otimes T_s(e_j)     \right)\ \ \ \ \ \ \ \ \ \ \ \ \ \ \ \ \ \ \ \text{by (1.1)}\\ 
&=&\left(M\otimes M\right)  \left(  T_s(e_i)\otimes   T_{t_s^2}(e_{\sigma'\sigma(k_s)})\otimes T_{t_s^1}(e_{\sigma'\sigma(k_s)})           \otimes T_s(e_j)     \right)\\
&=&  \left(\sum_{a\in\mathcal{O}_{i}^{s}\cap\mathcal{O}_{\sigma'\sigma(k_s)}^{t_s^2}}{e_a} \right) \otimes \left(\sum_{b\in\mathcal{O}_{\sigma'\sigma(k_s)}^{t_s^1}\cap\mathcal{O}_{j}^{s}}{e_b} \right) \\ 
&=& e_j\otimes e_i \ \ \ \ \ \ \ \ \ \ \ \ \ \ \ \ \ \ \ \ \ \ \ \ \ \ \ \ \ \ \ \ \ \ \ \ \ \ \ \ \ \ \ \ \ \ \ \  \ \ \ \ \ \ \ \text{by (1.2) et (1.3)}\\
\end{eqnarray*}

Moreover for all $i,j\in[0,p-1]$ we have
$$g_s(e_i\otimes e_j)=(\id\otimes M)(e_i\otimes T_s(e_i)\otimes e_j)=\delta_{j\in\mathcal{O}_{i}^{s}} (e_i\otimes e_j) $$

We end the proof by proving that $S=F\in C_{A(X)}$.

$\forall i,j\in[0,p-1]$, we denote $s_j$ the only one integer such that $j\in\mathcal{O}_{i}^{s_j}$ and we obtain
\begin{eqnarray*}
F(e_i\otimes e_j)&=&\sum_{s=0}^{\lambda}{f_s\circ g_s(e_i\otimes e_j)}\\
&=& \sum_{s=0}^{\lambda}{  \delta_{j\in\mathcal{O}_{i}^{s}} f_s(e_i\otimes e_j)      }\\
&=&f_{s_j}(e_i\otimes e_j) \\ 
&=& e_j\otimes e_i\\
\end{eqnarray*}
\end{proof}

The assumptions of proposition \ref{NoSym} may be weakened by using this following lemma \ref{LL3}, to obtain the theorem \ref{NoSym2}.

\begin{lem}\label{LL3} This two following propositions are equivalent:\smallskip

\begin{itemize}
\renewcommand{\labelitemi}{$\bullet$}
\item There exist $s_1,s_2\in [1,r]$  such that $\left|\mathcal{O}_0^{s_1}\cap\mathcal{O}_{1}^{s_2}\right| =1$\smallskip

\item For all $s\in[1,r]$, there exist $t_s^1,t_s^2\in [1,\lambda]$  such that, $\left|\mathcal{O}_0^{t_s^1}\cap\mathcal{O}_{y_s}^{t_s^2}\right| =1$
\end{itemize}
\end{lem}

\begin{proof}It's a direct consequence of proposition \ref{D1} and definition of orbitals $\mathcal{O}^s$ because for all $s\in[1,r]$, there exists $\sigma_s \in\mathrm{Aut}(X)$ such that $\sigma_s(0)=0$ and $\sigma_s(1)=y_s$.
\end{proof}

\begin{theorem}\label{NoSym2} Let $X$ be a circulant $p$-graph such that
$$\text{There exist}\ \  s_1,s_2\in[1,r] \text{ such that}\ \left|\mathcal{O}_0^{s_1}\cap\mathcal{O}_{1}^{s_2}\right| =1$$
Then  $X$ has no quantum symmetry.
\end{theorem}

\begin{remark} Using explicit description of $\mathcal{O}_{i}^{s}$ with the set $E$, the condition over $X$ in the theorem \ref{NoSym2} is equivalent to
$$\text{There exist}\ x,x'\in[1,p-1]\ \text{such that}\ \left|(xE)\cap(x'E+1)\right|=1$$
\end{remark}

 We denote $\Gamma_X=(\beta_{s_1s_2})$ the matrix with entries define by:
$$\beta_{s_1,s_2}=\left|\mathcal{O}_0^{s_1}\cap\mathcal{O}_{1}^{s_2}\right|=\left|(y_{s_1}E)\cap(y_{s_2}E+1)\right|$$

First the matrix $\Gamma_X$ is symmetrical since there exists $\sigma\in\Aut(X)$ such that $\sigma(0)=1$ and $\sigma(1)=0$. Then assumptions over $X$ in theorem \ref{NoSym2} are equivalent to say that at least one entry of $\Gamma_X$ is equal to $1$. Using this, we study quantum symmetries of $p$-graphs of types $4, 6, 8$ and $10$.

\subsection{Graphs of type $4$:}

\medskip

 By \ref{Borne0} we just need to study the case $p\le 36$. There is $4$ such graphs of type $4$ (counting only the one satisfying $S=E$). 
\begin{itemize}
\renewcommand{\labelitemi}{$\bullet$}
\item $X=C_{29}(12)$ ($r=7$):
$$\mathcal{O}_0^{1}=E=\{1,12,17,28\},\ \ \ \mathcal{O}_1^{2}=2E+1=\{3,6,25,28\}$$
$$\mathcal{O}_0^{1}\cap\mathcal{O}_1^{2}=\{28\}$$
so theorem \ref{NoSym2} holds.

\renewcommand{\labelitemi}{$\bullet$}
\item $X=C_{17}(4)$ ($r=4$):
$$\mathcal{O}_0^{1}=E=\{1,4,13,16\},\ \ \ \mathcal{O}_1^{2}=2E+1=\{3,9,10,16\}$$
$$\mathcal{O}_0^{1}\cap\mathcal{O}_1^{2}=\{16\}$$
so theorem \ref{NoSym2} holds.

\item $X=C_{13}(5)$ ($r=3$): 
$$\mathcal{O}_0^{1}=E=\{1,5,8,12\},\ \ \ \mathcal{O}_1^{2}=2E+1=\{3,4,11,12\}$$
$$\mathcal{O}_0^{1}\cap\mathcal{O}_1^{2}=\{12\}$$
so theorem \ref{NoSym2} holds.

\item $X=K_{5}$: the complete graph with $5$ vertices has quantum symmetries. 
\end{itemize}
\medskip

\subsection{Graphs of type $6$:}\medskip

 By \ref{Borne0} we just need to study the case $p\le 36$. There is $3$ such graphs of type $6$ (counting only the one satisfying $S=E$). 
\begin{itemize}
\renewcommand{\labelitemi}{$\bullet$}
\item $X=C_{31}(1,5,6)$ ($r=5$):
$$\mathcal{O}_0^{1}=E=\{1,5,6,25,26,30\},\ \ \ \mathcal{O}_1^{2}=2E+1=\{3,11,13,20,22,30\}$$
$$\mathcal{O}_0^{1}\cap\mathcal{O}_1^{2}=\{30\}$$
so theorem \ref{NoSym2} holds.

\renewcommand{\labelitemi}{$\bullet$}
\item $X=C_{19}(1,7,8)$ ($r=3$):
$$\mathcal{O}_0^{1}=E=\{1,7,8,11,12,18\},\ \ \ \mathcal{O}_1^{2}=2E+1=\{3,4,6,15,17,18\}$$
$$\mathcal{O}_0^{1}\cap\mathcal{O}_1^{2}=\{18\}$$
so theorem \ref{NoSym2} holds.

\item $X=C_{13}(3,4)$ ($r=2$): in this case we have
$$\mathcal{O}_0^{1}=E=\{1,3,4,9,10,12\},\ \ \ \mathcal{O}_0^{2}=2E=\{2,5,6,7,8,11\}$$
$$\mathcal{O}_1^{1}=E+1=\{0,2,4,5,10,11\},\ \ \ \mathcal{O}_1^{2}=2E+1=\{3,6,7,8,9,12\}$$
hence
$$\Gamma_X=\left( {
\begin{array}{cc}
2 & 3 \\
3 & 3 \\
\end{array}
} \right)$$
so theorem \ref{NoSym2} does not hold.

\item $X=K_{7}$: the complete graph with $7$ vertices has quantum symmetries. 
\end{itemize}
\smallskip

\subsection{Graphs of type $8$:}\medskip

 By \ref{Borne0} we just need to study the case $p\le 1296$. There is $48$ such graphs of type $8$ (counting only the one satisfying $S=E$). 
 Here is the study of the five which have less than $100$ vertices, with theorem \ref{NoSym2}.
\begin{itemize}
\renewcommand{\labelitemi}{$\bullet$}
\item $X=C_{97}(22,33,47)$ ($r=12$):
$$\mathcal{O}_0^{1}=E=\{1,22,33,47,50,64,75,96\},\ \ \ \mathcal{O}_1^{2}=2E+1=\{3,4,32,45,54,67,95,96\}$$
$$\mathcal{O}_0^{1}\cap\mathcal{O}_1^{2}=\{96\}$$
so theorem \ref{NoSym2} holds.

\renewcommand{\labelitemi}{$\bullet$}
\item $X=C_{89}(12,34,37)$ ($r=11$):
$$\mathcal{O}_0^{1}=E=\{1,12,34,37,52,55,77,88\},\ \ \ \mathcal{O}_1^{2}=2E+1=\{3,16,22,25,66,69,75,88\}$$
$$\mathcal{O}_0^{1}\cap\mathcal{O}_1^{2}=\{88\}$$
so theorem \ref{NoSym2} holds.

\renewcommand{\labelitemi}{$\bullet$}
\item $X=C_{73}(10,22,27)$ ($r=9$):
$$\mathcal{O}_0^{1}=E=\{1,10,22,27,46,51,63,72\},\ \ \ \mathcal{O}_1^{2}=2E+1=\{3,20,21,30,45,54,55,72\}$$
$$\mathcal{O}_0^{1}\cap\mathcal{O}_1^{2}=\{72\}$$
so theorem \ref{NoSym2} holds.

\item $X=C_{41}(3,9,14)$ ($r=5$):
$$\mathcal{O}_0^{2}=2E=\{2,6,13,18,23,28,35,39\},\ \ \ \mathcal{O}_1^{5}=8E+1=\{9,11,12,18,25,31,32,34\}$$
$$\mathcal{O}_0^{2}\cap\mathcal{O}_1^{5}=\{18\}$$
so theorem \ref{NoSym2} holds.

\item $X=C_{17}(2,4,8)$ ($r=2$): in this case we have
$$\mathcal{O}_0^{1}=E=\{1,2,4,8,9,13,15,16\},\ \ \ \mathcal{O}_0^{2}=3E=\{3,5,6,7,10,11,12,14\}$$
$$\mathcal{O}_1^{1}=E+1=\{0,2,3,5,9,10,14,16\},\ \ \ \mathcal{O}_1^{2}=3E+1=\{4,6,7,8,11,12,13,15\}$$
hence
$$\Gamma_X=\left( {
\begin{array}{cc}
3 & 4 \\
4 & 4 \\
\end{array}
} \right)$$
so theorem \ref{NoSym2} does not hold.
\end{itemize}\smallskip

The other graphs of type $8$ were compute the same way. For all of them theorem \ref{NoSym2} holds.

\medskip

\subsection{Graphs of type $10$:}\medskip

 By \ref{Borne0} we just need to study the case $p\le 1296$. There is $51$ such graphs of type $10$ (counting only the one satisfying $S=E$). 
 Here is the study of the five which have less than $100$ vertices, with theorem \ref{NoSym2}.
\begin{itemize}
\renewcommand{\labelitemi}{$\bullet$}
\item $X=C_{71}(5,14,17,25)$ ($r=7$):
$$\mathcal{O}_0^{1}=E=\{1,5,14,17,25,46,54,57,66,70\}$$
$$\mathcal{O}_1^{2}=2E+1=\{3,11,22,29,35,38,44,51,62,70\}$$
$$\mathcal{O}_0^{1}\cap\mathcal{O}_1^{2}=\{70\}$$
so theorem \ref{NoSym2} holds.

\item $X=C_{61}(3,9,20,27,34,41,52,58,60)$ ($r=6$):
$$\mathcal{O}_0^{2}=2E=\{2,6,7,18,21,40,43,54,55,59\}$$ 
$$\mathcal{O}_1^{4}=4E+1=\{5,13,15,20,26,37,43,48,50,58\}$$
$$\mathcal{O}_0^{2}\cap\mathcal{O}_1^{4}=\{43\}$$
so theorem \ref{NoSym2} holds.

\item $X=C_{41}(4,10,16,18)$ ($r=4$): we have
$$\Gamma_X=\left( {
\begin{array}{cccc}
0 & 3 & 2 & 4 \\
3 & 3 & 2 & 2 \\
2 & 2 & 4 & 2 \\
4 & 2 & 2 & 2 \\
\end{array}
} \right)$$

\item $X=C_{31}(2,4,8,15)$ ($r=3$): we have
$$\Gamma_X=\left( {
\begin{array}{ccc}
3 & 4 & 2  \\
4 & 2 & 4  \\
2 & 4 & 4  \\
\end{array}
} \right)$$

\item $X=K_{11}$: the complete graph with $11$ vertices has quantum symmetries. 
\end{itemize}\medskip

The other graphs of type $10$ with $p\geqslant 71$ were compute the same way. For all of them, theorem \ref{NoSym2} holds.\medskip

The two next section are devoted to the study of $C_{13}(3,4)$ and $C_{17}(2,4,8)$. In section $7$ and $8$ we proves that $C_{13}(3,4)$ and $C_{17}(2,4,8)$ have no quantum symmetry. Then all the $p$-graphs of type $4$, $6$ or $8$ (except complete graphs) have no quantum symmetry. Is it the case for bigger type ? It's an open question.

\medskip

\section{$C_{13}(3,4)$ has no quantum symmetry}

The graph $C_{13}(3,4)$ is $\mathcal{B}$-clos (by \ref{pcoh}) but don't satisfy assumptions of theorem  \ref{NoSym2}. Indeed, here are its orbitals:\medskip

 \ \ \ \ \ \ \ \ \ \ \ \ \ \ \ \ \ \ \  \ \ \ \ \ \begin{tabular}{ | c | c | c |  }
 \hline                 
   $i$ & $\mathcal{O}_{i}^{1}$ &  $\mathcal{O}_{i}^{2}$  \tabularnewline \hline
$0$ & $\{1,3,4,9,10,12\}$ & $\{2,5,6,7,8,11\}$ \tabularnewline \hline
$1$ & $\{2,4,5,10,11,0\}$ & $\{3,6,7,8,9,12\}$ \tabularnewline \hline
$2$ & $\{3,5,6,11,12,1\}$ & $\{4,7,8,9,10,0\}$ \tabularnewline \hline
$3$ & $\{4,6,7,12,0,2\}$ & $\{5,8,9,10,11,1\}$ \tabularnewline \hline
$4$ & $\{5,7,8,0,1,3\}$ & $\{6,9,10,11,12,2\}$ \tabularnewline \hline
$5$ & $\{6,8,9,1,2,4\}$ & $\{7,10,11,12,0,3\}$ \tabularnewline \hline
$6$ & $\{7,9,10,2,3,5\}$ & $\{8,11,12,0,1,4\}$ \tabularnewline
\hline
$7$ & $\{8,10,11,3,4,6\}$ & $\{9,12,0,1,2,5\}$ \tabularnewline
\hline
$8$ & $\{9,11,12,4,5,7\}$ & $\{10,0,1,2,3,6\}$ \tabularnewline
\hline
$9$ & $\{10,12,0,5,6,8\}$ & $\{11,1,2,3,4,7\}$ \tabularnewline
\hline
$10$ & $\{11,0,1,6,7,9\}$ & $\{12,2,3,4,5,8\}$ \tabularnewline
\hline
$11$ & $\{12,1,2,7,8,10\}$ & $\{0,3,4,5,6,9\}$ \tabularnewline
\hline
$12$ & $\{0,2,3,8,9,11\}$ & $\{1,4,5,6,7,10\}$ \tabularnewline
 \hline
 \end{tabular}\bigskip

 We remark that elements of $\mathcal{O}_{i}^{1}$ are the neighbors of $i$ and the elements of $\mathcal{O}_{i}^{2}$ are the vertices non connected to $i$ in $C_{13}(3,4)$. It means that for every two couples $(i,j)$ and $(k,l)$ of vertices of same nature (connected or not) there exists an automorphism $\sigma\in\mathrm{Aut}(C_{13}(3,4))$ such that: 
 $$\sigma(i)=k\ \ \text{and}\ \ \sigma(j)=l$$

\begin{defin}We denote $D=M\circ (\id\otimes M)$ the function of $C_{A(C_{13}(3,4))}(3,1)$. 

Then for all $i,j,k\in[0,12]$, we have:
 $$D(e_i\otimes e_j\otimes e_k)=\delta_{i,j}\delta_{i,k}e_i$$
\end{defin} \medskip

 \begin{lem}\label{342} There exists an intertwinner $G_1\in C_{A(C_{13}(3,4))}(2,1)$ such that:
 $$G_1(e_0\otimes e_2)=e_1$$
 
 \end{lem}

 \begin{proof} We put $H_1=M\circ (T_1\otimes T_1)$ and $H_4=M\circ (T_2\otimes T_2)$, then:
 \begin{eqnarray*}
H_1(e_0\otimes e_2)&=&M\circ (T_1\otimes T_1)(e_0\otimes e_2)\\
&=&M\left((e_1+e_3+e_4+e_9+e_{10}+e_{12})\otimes(e_1+e_3+e_5+e_6+e_{11}+e_{12})\right)\\
&=& e_1+e_3+e_{12}\\
H_4(e_0\otimes e_2)&=&M\circ (T_2\otimes T_2)(e_0\otimes e_2)\\
&=&M\left((e_2+e_5+e_6+e_7+e_8+e_{11})\otimes(e_0+e_4+e_7+e_8+e_9+e_{10})\right)\\
&=& e_7+e_8\\
T_2(e_7+e_8)&=&2e_0+2e_1+2e_2+e_3+e_5+e_6+e_9+e_{10}+e_{12}
\end{eqnarray*}

We define the intertwinner $G\in C_{A(C_{13}(3,4))}(2,1)$  by
$$G=D \circ (T_1\otimes T_2\otimes T_1)\circ (\id\otimes H_4\otimes \id)\circ (M^*\otimes M^*)$$

So we get:
\begin{eqnarray*}
G(e_0\otimes e_2)&=&D \circ (T_1\otimes T_2\otimes T_1)\circ (\id\otimes H_4\otimes \id)\circ (M^*\otimes M^*)(e_0\otimes e_2)\\
&=&D \circ (T_1\otimes T_2\otimes T_1)\left(e_0\otimes (e_7+e_{8})\otimes e_2\right)\\
&=& D\left((e_1+e_3+e_4+e_9+e_{10}+e_{12})\otimes T_2(e_7+e_{8})\otimes (e_1+e_3+e_5+e_6+e_{11}+e_{12}) \right)\\
&=& 2e_1+e_3+e_{12}
\end{eqnarray*}

Finally with $G_1=G-H_1\in  C_{A(C_{13}(3,4))}(2,1)$ we have:
$$G_1(e_0\otimes e_2)= 2e_1+e_3+e_{12}-(e_1+e_3+e_{12})=e_1$$

 as required.
 \end{proof}

 \begin{lem}\label{341} For all $k\in[0,12]$  there exists an intertwinner $F_k\in C_{A(C_{13}(3,4))}(2,1)$ such that:
 $$F_k(e_0\otimes e_1)=e_k$$
 
 \end{lem}

 \begin{proof}\

First we define \fbox{$F_0=M\circ(\id\otimes (U\circ U^*))$} and \fbox{$F_1=M\circ( (U\circ U^*))\otimes\id)$} to have:
 $$F_0(e_0\otimes e_1)=e_0,\ \ \text{ and }\ \  F_1(e_0\otimes e_1)=e_1$$
 
 Then, we consider the following intertwinners of $C_{A(C_{13}(3,4))}(2,1)$:
 $$H_1=M\circ (T_1\otimes T_1),\ \ H_2=M\circ (T_1\otimes T_2),\ \ H_3=M\circ (T_2\otimes T_1),\ \ H_4=M\circ (T_2\otimes T_2)$$
and we obtain:
 $$H_1(e_0\otimes e_1)=e_4+e_{10},\ \ \ H_2(e_0\otimes e_1)=e_3+e_9+e_{12}$$
   $$H_3(e_0\otimes e_1)=e_2+e_5+e_{11},\ \ \ H_4(e_0\otimes e_1)=e_6+e_{7}+e_8$$
 
We remark that: $T_1(e_4+e_{10})=2e_0+2e_1+2e_7+e_3+e_5+e_6+e_8+e_9+e_{11}$, so we obtain:
 \begin{eqnarray*}
& &D \circ (T_2\otimes T_1\otimes T_2)\circ (\id\otimes H_1\otimes \id)\circ (M^*\otimes M^*)(e_0\otimes e_1)\\
&=&D \circ (T_2\otimes T_1\otimes T_2)\left(e_0\otimes (e_4+e_{10})\otimes e_1\right)\\
&=& D\left((e_2+e_5+e_6+e_7+e_8+e_{11})\otimes T_1(e_4+e_{10})\otimes (e_3+e_6+e_7+e_8+e_9+e_{12}) \right)\\
&=& 2e_7+e_6+e_8
\end{eqnarray*}
 hence \fbox{$F_7=D \circ (T_2\otimes T_1\otimes T_2)\circ (\id\otimes H_1\otimes \id)\circ (M^*\otimes M^*)-H_4$} works. \smallskip

Then: \fbox{$F_3=D\circ (T_1\otimes T_1\otimes T_2)\circ (\id\otimes F_7\otimes \id)\circ (M^*\otimes M^*)$} and 

\fbox{$F_{11}=D\circ (T_2\otimes T_1\otimes T_1)\circ (\id\otimes F_7\otimes \id)\circ (M^*\otimes M^*)$} work since:
\begin{eqnarray*}
F_3(e_0\otimes e_1)&=&D\circ (T_1\otimes T_1\otimes T_2)\circ (\id\otimes F_7\otimes \id)(e_0\otimes e_0\otimes e_1\otimes  e_1)\\
&=&D\circ (T_1\otimes T_1\otimes T_2)(e_0\otimes e_7\otimes e_1)\\
&=& D\left(T_1(e_0)\otimes (e_3+e_4+e_6+e_7+e_{10}+e_{11})\otimes T_2(e_1)\right)=e_3
\end{eqnarray*}
and
\begin{eqnarray*}
F_11(e_0\otimes e_1)&=&D\circ (T_2\otimes T_1\otimes T_1)\circ (\id\otimes F_7\otimes \id)(e_0\otimes e_0\otimes e_1\otimes  e_1)\\
&=&D\circ (T_2\otimes T_1\otimes T_1)(e_0\otimes e_7\otimes e_1)\\
&=& D\left(T_2(e_0)\otimes (e_3+e_4+e_6+e_7+e_{10}+e_{11})\otimes T_1(e_1)\right)=e_{11}
\end{eqnarray*}
 
 Finally we use:
 \begin{eqnarray*}
T_2(e_4+e_{10})&=&2e_2+2e_{12}+e_3+e_4+e_5+e_6+e_8+e_9+e_{10}+e_{11}\\
T_2(e_9)&=&e_1+e_{2}+e_3+e_4+e_7+e_{11}\\
T_2(e_3)&=&e_1+e_{5}+e_8+e_9+e_{10}+e_{11}
\end{eqnarray*}
 
 to find the other intertwinners we use: \medskip
 
\fbox{$F_2=D \circ (T_2\otimes T_2\otimes T_1)\circ (\id\otimes H_1\otimes \id)\circ (M^*\otimes M^*)-H_3$}.\smallskip

\fbox{$F_{12}=D \circ (T_1\otimes T_2\otimes T_2)\circ (\id\otimes H_1\otimes \id)\circ (M^*\otimes M^*)-H_2$}.\smallskip

\fbox{$F_{5}=H_3-F_2-F_{11}$}\ \ \ \ \  \fbox{$F_{9}=H_2-F_3-F_{12}$}.\smallskip

\fbox{$F_4=D\circ (T_1\otimes T_2\otimes T_1)\circ (\id\otimes F_9\otimes \id)\circ (M^*\otimes M^*)$}\smallskip

\fbox{$F_8=D\circ (T_2\otimes T_2\otimes T_2)\circ (\id\otimes F_{3}\otimes \id)\circ (M^*\otimes M^*)$}\smallskip

\fbox{$F_{10}=H_1-F_4$}\ \ \ \ \  \fbox{$F_{6}=H_4-F_7-F_{8}$}.\smallskip

These functions well satisfy $F_k(e_0\otimes e_1)=e_k$ as required
 \end{proof}
 \smallskip
\smallskip \smallskip

 \begin{theorem} The graph $C_{13}(3,4)$ has no quantum symmetry.
 \end{theorem}\medskip

  \begin{proof} We first remark a property about the functions $F_k$. For any \underline{connected} vertices $i$ and $j$, as $0$ and $1$ are connected there exists $\sigma\in\mathrm{Aut}(C_{13}(3,4))$ such that: $\sigma(0)=i$ and $\sigma(1)=j$. Hence, for all  $k\in[0,12]$:
$$F_k(e_i\otimes e_j)=F_k(e_{\sigma(0)}\otimes e_{\sigma(1)})=P_{\sigma}\circ F_k(e_0\otimes e_1)=e_{\sigma(k)}\ \ \ \ (*)$$

Then we need to built new intertwinners:

\noindent \begin{itemize}
\renewcommand{\labelitemi}{$\bullet$}
\item  The vertex $4$ is a neighbor of $0$ and $1$ so we can consider $\mu, \mu'\in\mathrm{Aut}(C_{13}(3,4))$ such that:
$$\mu(0)=0,\ \ \mu(1)=4,\ \ \mu'(0)=4,\ \ \mu'(1)=1$$
and we denote by $x, y\in[0,12]$ the vertices such that $\mu(x)=1$ and $\mu'(y)=0$ to define:
\begin{center}
\fbox{$F=(F_{x}\otimes F_{y})\circ (\id\otimes (M^*\circ F_4)\otimes \id)\circ (M^*\otimes M^*)$} 
\end{center} 
So by $(*)$ we have:
\begin{eqnarray*}
F(e_0\otimes e_1)&=&(F_{x}\otimes F_{y})\circ (\id\otimes (M^*\circ F_4)\otimes \id)(e_0\otimes e_0\otimes e_1\otimes  e_1)\\
&=&(F_{x}\otimes F_{y})(e_0\otimes e_4\otimes e_4 \otimes  e_1)=F_x(e_0\otimes e_4)\otimes F_y(e_4\otimes e_1)\\
&=& e_{\mu(x)}\otimes e_{\mu'(y)}=e_1\otimes e_0
\end{eqnarray*}
Then if $i$ and $j$ are connected, with $\sigma$ such that $\sigma(0)=i$ and $\sigma(1)=j$ we have:
$$F(e_i\otimes e_j)=F(e_{\sigma(0)}\otimes e_{\sigma(1)})=P_{\sigma}^{\otimes 2}\circ F(e_0\otimes e_1)=e_{\sigma(1)}\otimes e_{\sigma(0)}=e_j\otimes e_i$$

\item  The vertex $1$ is a neighbor of $0$ and $2$ so we can consider $\tau\in\mathrm{Aut}(C_{13}(3,4))$ such that:
$$\tau(0)=1\ \text{ and }\  \tau(1)=2$$ et we denote by $z\in[|0,12|]$ the integer such that $\tau(z)=0$ to define:
\begin{center}
\fbox{$G=(F_2\otimes F_{z})\circ (\id\otimes (M^*\circ G_1)\otimes \id)\circ (M^*\otimes M^*)$}
\end{center}

 where $G_1$ is the function of the lemma \ref{342}. Then by $(*)$, we have:
\begin{eqnarray*}
G(e_0\otimes e_2)&=&(F_2\otimes F_{z})\circ (\id\otimes (M^*\circ G_1)\otimes \id)(e_0\otimes e_0\otimes e_2\otimes  e_2)\\
&=&(F_2\otimes F_{z})(e_0\otimes e_1\otimes e_1 \otimes  e_2)=F_2(e_0\otimes e_1)\otimes F_z(e_1\otimes e_2)\\
&=& e_{2}\otimes e_{\tau(z)}=e_2\otimes e_0
\end{eqnarray*}

If $i$ and $j$ are not connected, with $\sigma$ such that $\sigma(0)=i$ and $\sigma(2)=j$ we have:
$$G(e_i\otimes e_j)=G(e_{\sigma(0)}\otimes e_{\sigma(2)})=P_{\sigma}^{\otimes 2}\circ G(e_0\otimes e_2)=e_{\sigma(2)}\otimes e_{\sigma(0)}=e_j\otimes e_i$$
\item For $s\in\{0,1,2\}$ we define:
\begin{center}
\fbox{$g_s:=(\id\otimes M)\circ(\id\otimes T_s\otimes\id)\circ(M^*\otimes \id)\in C_{A(C_{13}(3,4))}(2,2)$}
\end{center}
This intertwinner are easy because:
$$g_0(e_i\otimes e_j)=\mathds{1}_{i=j}(e_i\otimes e_i),\ \  g_1(e_i\otimes e_j)=\mathds{1}_{i\sim j}(e_i\otimes e_j),\ \ g_0(e_i\otimes e_j)=\mathds{1}_{i\not\sim j}(e_i\otimes e_j)$$

\item To finish the proof we use the last intertwinner:
\begin{center}
\fbox{$H:=g_0+F\circ g_1+G\circ g_2\in C_{A(C_{13}(3,4))}\left( 2,2   \right) $}
\end{center}
We have:
$$H(e_i\otimes e_j)=\left\lbrace\begin{array}{l}
e_i\otimes e_i \ \ \ \ \ \ \ \ \ \text{ if }i=j \\
F(e_i\otimes e_j) \ \ \ \ \text{ if }i\sim j \\
G(e_i\otimes e_j) \ \ \ \ \text{ if }i\not\sim j \\
\end{array}
\right\rbrace= e_j\otimes e_i$$
hence $S=H\in C_{A(C_{13}(3,4))}$ and the graph $C_{13}(3,4)$ has no quantum symmetry by corollary \ref{CorS}.
\end{itemize}

 \end{proof}

This result gives us also the complete study of quantum symmetries of vertex-transitive graphs of order $13$:

\begin{center}
\begin{tabular}{   | l| l | l | }
 \hline                 
    $X$ & $\mathrm{Aut}(X)$ &  $\mathbb{G}_X$  \tabularnewline \hline   
 $K_{13}$ & $\mathbb{S}_{13}$ & $S_{13}^+$ \tabularnewline \hline
 $C_{13}$ & $\mathcal{D}_{13}$ & $\mathcal{D}_{13}$ \tabularnewline \hline
 $C_{13}(2)$ & $\mathcal{D}_{13}$ & $\mathcal{D}_{13}$ \tabularnewline \hline
 $C_{13}(2,5)$ & $\mathcal{D}_{13}$ & $\mathcal{D}_{13}$ \tabularnewline \hline
 $C_{13}(2,6)$ & $\mathcal{D}_{13}$ & $\mathcal{D}_{13}$ \tabularnewline \hline
 $C_{13}(3)$ & $\mathcal{D}_{13}$ & $\mathcal{D}_{13}$ \tabularnewline \hline
 $C_{13}(5)$ & $G_1:=\Z_{13}\rtimes \Z_{4}$ & $G_1$ \tabularnewline \hline
 $C_{13}(3,4)$ & $G_2:=\Z_{13}\rtimes \Z_{6}$ & $G_2$  \\
 \hline  
 \end{tabular}
  \end{center}\medskip

To complete the work of \cite{BaBi11} we should also study the graphs of order $12$.

\medskip

\section{$C_{17}(2,4,8)$ has no quantum symmetry}

The graph $C_{17}(2,4,8)$ is $\mathcal{B}$-clos (by \ref{pcoh}) but don't satisfy assumptions of theorem  \ref{NoSym2}  Indeed, here are its orbitals:\medskip

 \ \ \ \ \ \ \ \ \ \ \ \ \ \ \ \ \ \ \  \ \ \ \ \ \begin{tabular}{ | c | c | c |  }
 \hline                 
   $i$ & $\mathcal{O}_{i}^{1}$ &  $\mathcal{O}_{i}^{2}$  \tabularnewline \hline
$0$ & $\{1,2,4,8,9,13,15,16\}$ & $\{3,5,6,7,10,11,12,14\}$ \tabularnewline \hline
$1$ & $\{2,3,5,9,10,14,16,0\}$ & $\{4,6,7,8,11,12,13,15\}$ \tabularnewline \hline
$2$ & $\{3,4,6,10,11,15,0,1\}$ & $\{5,7,8,9,12,13,14,16\}$ \tabularnewline \hline
$3$ & $\{4,6,7,12,0,2\}$ & $\{5,8,9,10,11,1\}$ \tabularnewline \hline
$4$ & $\{5,7,8,0,1,3\}$ & $\{6,9,10,11,12,2\}$ \tabularnewline \hline
$5$ & $\{6,8,9,1,2,4\}$ & $\{7,10,11,12,0,3\}$ \tabularnewline \hline
$6$ & $\{7,9,10,2,3,5\}$ & $\{8,11,12,0,1,4\}$ \tabularnewline
\hline
$7$ & $\{8,10,11,3,4,6\}$ & $\{9,12,0,1,2,5\}$ \tabularnewline
\hline
$8$ & $\{9,11,12,4,5,7\}$ & $\{10,0,1,2,3,6\}$ \tabularnewline
\hline
$9$ & $\{10,12,0,5,6,8\}$ & $\{11,1,2,3,4,7\}$ \tabularnewline
\hline
$10$ & $\{11,0,1,6,7,9\}$ & $\{12,2,3,4,5,8\}$ \tabularnewline
\hline
$11$ & $\{12,1,2,7,8,10\}$ & $\{0,3,4,5,6,9\}$ \tabularnewline
\hline
$12$ & $\{0,2,3,8,9,11\}$ & $\{1,4,5,6,7,10\}$ \tabularnewline
 \hline
 \end{tabular}\bigskip

As in $C_{13}(3,4)$, the elements of $\mathcal{O}_{i}^{1}$ are the neighbors of $i$ and the elements of $\mathcal{O}_{i}^{2}$ are the vertices non connected to $i$ in $C_{13}(3,4)$. For every two couples $(i,j)$ and $(k,l)$ of vertices of same nature (connected or not) there exists an automorphism $\sigma\in\mathrm{Aut}(C_{17}(2,4,8))$ such that: 
 $$\sigma(i)=k\ \ \text{and}\ \ \sigma(j)=l$$

\begin{defin}We denote $D=M\circ (\id\otimes M)$ the function of $C_{A(C_{17}(2,4,8))}(3,1)$. 

Then for all $i,j,k\in[0,16]$, we have:
 $$D(e_i\otimes e_j\otimes e_k)=\delta_{i,j}\delta_{i,k}e_i$$
 
\end{defin} \medskip

 \begin{lem}\label{248_1} For all $k\in  \mathcal{O}_{0}^{1}\cap \mathcal{O}_{1}^{1}=\{2,9,16\}$, there exists an intertwinner $F_k\in C_{A(C_{17}(2,4,8))}(2,1)$ such that:
 $$F_k(e_0\otimes e_1)=e_k$$
 \end{lem}

 \begin{proof} For all $i,j\in\{0,1\}$, we define: $H_{i,j}=M\circ (T_i\otimes T_j)\in C_{A(C_{17}(2,4,8))}(2,1)$ then we have:
 $$H_{1,1}(e_0\otimes e_1)=e_2+e_9+e_{16},\ \ \ \ \ \ H_{1,2}(e_0\otimes e_1)=e_4+e_8+e_{13}+e_{15}$$
  $$H_{2,1}(e_0\otimes e_1)=e_3+e_5+e_{10}+e_{14},\ \ \ \ \ \ H_{2,2}(e_0\otimes e_1)=e_6+e_7+e_{11}+e_{12}$$
Then for $H=T_1\circ H_{1,1}$, we have:
$$H(e_0\otimes e_1)=3(e_0+e_1)+2(e_8+e_{15}+e_3+e_{10}+e_7+e_{11})+(e_4+e_{13}+e_5+e_{14}+e_6+e_{12})$$
 
 Now we can see that if $S\in C_{A(C_{17}(2,4,8))}(2,1)$ then $\left((T_i\otimes S\otimes T_j)\circ (M^*\circ M^*)\right)(e_0\otimes e_1)$ gives us only the $"e_k"$ of $S(e_0\otimes e_1)$ that are also in $H_{i,j}(e_0\otimes e_1)$. For example:
 
 $$G_2:=D\circ (T_2\otimes H\otimes T_1)\circ (M^*\circ M^*):e_0\otimes e_1\mapsto 2(e_3+e_{10})+e_5+e_{14}$$
  $$G_9:=D\circ (T_2\otimes H\otimes T_2)\circ (M^*\circ M^*):e_0\otimes e_1\mapsto 2(e_7+e_{11})+e_6+e_{12}$$
 $$G_{16}:=D\circ (T_1\otimes H\otimes T_2)\circ (M^*\circ M^*):e_0\otimes e_1\mapsto 2(e_8+e_{15})+e_4+e_{13}$$

  Then we define: 
  $$S_2=T_1\circ (G_2-H_{2,1}),\ \ \ S_9=T_1\circ (G_9-H_{2,2}),\ \ \ S_{16}=T_1\circ (G_{16}-H_{1,2})$$
  and we have:
$$S_2(e_0\otimes e_1)=T_1\left(e_3+e_{10}\right)=2e_2+e_9+e_{16}+\dots \text{"other }e_k"\dots $$ 
$$S_9(e_0\otimes e_1)=T_1\left(e_3+e_{10}\right)=e_2+2e_9+e_{16}+\dots \text{"other }e_k"\dots $$ 
$$S_2(e_0\otimes e_1)=T_1\left(e_3+e_{10}\right)=e_2+e_9+2e_{16}+\dots \text{"other }e_k"\dots $$

  To finsih the proof we define this following intertwiners of $C_{A(C_{17}(2,4,8))}(2,1)$:
  $$F_2=D\circ (T_1\otimes S_2\otimes T_1)\circ (M^*\circ M^*)-H_{1,1}$$
  $$F_9=D\circ (T_1\otimes S_9\otimes T_1)\circ (M^*\circ M^*)-H_{1,1}$$
  $$F_2=D\circ (T_1\otimes S_{16}\otimes T_1)\circ (M^*\circ M^*)-H_{1,1}$$
 to obtain:
$$F_2(e_0\otimes e_1)=(2e_2+e_9+e_{16})-(e_2+e_9+e_{16})=e_2$$
$$F_9(e_0\otimes e_1)=(e_2+2e_9+e_{16})-(e_2+e_9+e_{16})=e_9$$
$$F_{16}(e_0\otimes e_1)=(e_2+e_9+2e_{16})-(e_2+e_9+e_{16})=e_{16}$$
 \end{proof}

 \begin{theorem} The graphe $C_{17}(2,4,8)$ has no quantum symmetry.
 \end{theorem}\medskip

  \begin{proof} As for $C_{13}(3,4)$  we check that for any \underline{connected} vertices $i$ and $j$, as $0$ and $1$ are connected there exists $\sigma\in\mathrm{Aut}(C_{17}(2,4,8))$ such that: $\sigma(0)=i$ and $\sigma(1)=j$. Hence, for all  $k\in\{2,9,16\}$:
$$F_k(e_i\otimes e_j)=F_k(e_{\sigma(0)}\otimes e_{\sigma(1)})=P_{\sigma}\circ F_k(e_0\otimes e_1)=e_{\sigma(k)}\ \ \ \ (*)$$

We recall that $\{2,9,10\}=\mathcal{O}_{0}^{1}\cap \mathcal{O}_{1}^{1}$. so
$$\sigma\left(\{2,9,10\}\right)=\mathcal{O}_{\sigma(0)}^{1}\cap \mathcal{O}_{\sigma(1)}^{1}=\mathcal{O}_{i}^{1}\cap \mathcal{O}_{j}^{1}$$ 

by proposition \ref{D1}.

The vertices $0$ and $2$ are connected and $1\in \mathcal{O}_{0}^{1}\cap \mathcal{O}_{2}^{1}$ so there exists $x\in\{2,9,16\}$ such that $F_x(e_0\otimes e_2)=e_1$. The same way, as $1$ and $2$ are connected and $0\in \mathcal{O}_{1}^{1}\cap \mathcal{O}_{2}^{1}$, there exists $y\in\{2,9,16\}$ such that $F_y(e_2\otimes e_1)=e_0$

Then we define: 
$$L=(F_{x}\otimes F_{y})\circ (\id\otimes (M^*\circ F_2)\otimes \id)\circ (M^*\otimes M^*)\in C_{A(C_{17}(2,4,8))}(2,2)$$
So we have:
\begin{eqnarray*}
L(e_0\otimes e_1)&=&(F_{x}\otimes F_{y})\circ (\id\otimes (M^*\circ F_2)\otimes \id)(e_0\otimes e_0\otimes e_1\otimes  e_1)\\
&=&(F_{x}\otimes F_{y})(e_0\otimes e_2\otimes e_2 \otimes  e_1)=F_x(e_0\otimes e_2)\otimes F_y(e_2\otimes e_1)=e_1\otimes e_0
\end{eqnarray*}

Then if $i$ and $j$ are connected, with $\sigma$ such that $\sigma(0)=i$ and $\sigma(1)=j$ we have:
$$L(e_i\otimes e_j)=L(e_{\sigma(0)}\otimes e_{\sigma(1)})=P_{\sigma}^{\otimes 2}\circ L(e_0\otimes e_1)=e_{\sigma(1)}\otimes e_{\sigma(0)}=e_j\otimes e_i$$

As for $C_{13}(3,4)$, for $s\in\{0,1,2\}$ we define:
\begin{center}
\fbox{$g_s:=(\id\otimes M)\circ(\id\otimes T_s\otimes\id)\circ(M^*\otimes \id)\in C_{A(C_{17}(2,4,8))}(2,2)$}
\end{center}
This $g_s$ are as follows:
$$g_0(e_i\otimes e_j)=\mathds{1}_{i=j}(e_i\otimes e_i),\ \  g_1(e_i\otimes e_j)=\mathds{1}_{i\sim j}(e_i\otimes e_j),\ \ g_0(e_i\otimes e_j)=\mathds{1}_{i\not\sim j}(e_i\otimes e_j)$$

Then  for $L'=L\circ g_1$ we have:
$$L'(e_i\otimes e_j)=\left\lbrace\begin{array}{l}
L(e_i\otimes e_j) \ \ \ \ \text{ if }i\sim j \\
\ \ \ \ \ \ 0 \ \ \ \ \ \ \ \ \ \ \text{ if }i\not\sim j \\
\end{array}
\right\rbrace= \left\lbrace\begin{array}{l}
e_j\otimes e_i \ \ \ \ \text{ if }i\sim j \\
\ \ \ \ 0  \ \ \ \ \ \ \ \text{ if }i\not\sim j \\
\end{array}
\right\rbrace$$

We can now use the fact that $C_{17}(2,4,8)$ is self adjoint, so as $C_{A((C_{17}(2,4,8))^c)}(2,2)=C_{A(C_{17}(2,4,8))}(2,2)$ there exists $L''\in C_{A(C_{17}(2,4,8))}(2,2)$ such that:
$$L''(e_i\otimes e_j)= \left\lbrace\begin{array}{l}
e_j\otimes e_i \ \ \ \ \text{ if }i\not\sim j\ \text{and} \ i\ne j  \\
\ \ \ \ 0  \ \ \ \ \ \ \ \text{ if }i\sim j \text{ or }i=j \\
\end{array}
\right\rbrace$$

To finish the proof we use the last intertwinner:
\begin{center}
\fbox{$G:=g_0+L'+L''\in C_{A(C_{17}(2,4,8))}\left( 2,2   \right) $}
\end{center}
We have, $G(e_i\otimes e_j)= e_j\otimes e_i$, hence $S=G\in C_{A(C_{17}(2,4,8))}$ and the graph $C_{17}(2,4,8)$ has no quantum symmetry by corollary \ref{CorS}.
 \end{proof}

It should be interested to study the same way $C_{41}(4,10,16,18)$ and $C_{31}(2,4,8,15)$ to check if all the circulant $p$-graphs of order $10$ have no quantum symmetry.

\section{Applications in the general case}\label{PartCasgeneral}

\begin{theorem}\label{NoSymG}Let $X$ be a  $\mathcal{B}$-clos graph, such that\smallskip

For every $s\in[1,r]$, there exist $ j_s\in\mathcal{O}_0^s$, $k_s\in[0,n-1]$ and $t_{s}^1,t_s^2,t_s^3,t_s^4,t_s^5\in[0,r]$ such that
$$\mathcal{O}_0^{t_{s}^1}\cap\mathcal{O}_{j_s}^{t_{s}^2}=\{k_s\},\ \ \ \ \ \mathcal{O}_0^{s}\cap\mathcal{O}_{k_s}^{t_{s}^4}=\{j_s\},\ \ \ \ \ \mathcal{O}_{k_s}^{t_s^5}\cap\mathcal{O}_{j_s}^{t_{s}^3}=\{0\}$$
Then $X$ has no quantum symmetry.
\end{theorem}

\begin{remark} It is equivalent to replace the assumption  "there exists $j_s\in\mathcal{O}_0^s$"  by "for every $j_s\in\mathcal{O}_0^s$" because if it holds for one $j_s$ then, applying elements of $\Aut_0(X)$ to the third wishes equality we get the same property for every elements of the orbit of $j_s$ under action of $\Aut_0(X)$, which is $\mathcal{O}_0^s$.
\end{remark}

\begin{proof} It's nearly the same demonstration of proposition \ref{NoSym} with some quick changes. We put $t_0^1=t_0^2=t_0^3=t_0^4=t_0^5=0$.
Then we define the following intertwinners of $C_{A(X)}$ for every $s\in[0,r]$.
\begin{eqnarray*}
h_s&:=&\left(M\otimes M\right) \circ \left(\id\otimes T_{t_s^4}\otimes T_{t_s^5}\otimes \id\right)\in C_{A(X)}(4,2)\\
f_{s}&:=&h_s\circ \left(\id\otimes (M^*\circ M)\otimes\id\right)\circ \left(T_s\otimes T_{t_s^1}\otimes T_{t_s^2}\otimes T_{t_s^3}\right)  \circ \left(M^*\otimes M^*\right) \in C_{A(X)}\left( 2,2   \right)\\
g_s&:=&(\id\otimes M)\circ(\id\otimes T_s\otimes\id)\circ(M^*\otimes \id)\in C_{A(X)}(2,2)\\
F&:=&\sum_{s=0}^r{f_s\circ g_s}\in C_{A(X)}\left( 2,2   \right) 
\end{eqnarray*}

We check that $f_0=\id_{(\C^n)^{\otimes 2}}$, then for  $s\in[1,r]$, we have
$$\mathcal{O}_{0}^{t_s^1}\cap\mathcal{O}_{j_s}^{t_s^2}=\{k_s\}$$

We now consider two vertices $i$ and $j$ such that $j\in\mathcal{O}_{i}^{s}$. Let $\sigma\in\Aut(X)$ such that $\sigma(0)=i$. We know that $j_s\in\mathcal{O}_{0}^{s}$ hence $\sigma(j_s)\in\mathcal{O}_{i}^{s}$. So there exists $\sigma'\in\Aut_i(X)$ such that $\sigma'(\sigma(j_s))=j$.  Hence,  using $(iii)$ of proposition \ref{D1} we have:
$$\{\sigma'\sigma(k_s)\}=\sigma'\sigma(\mathcal{O}_{0}^{t_s^1}\cap\mathcal{O}_{j_s}^{t_s^2})=
\mathcal{O}_{\sigma'\sigma(0)}^{t_s^1}\cap\mathcal{O}_{\sigma'\sigma(j_s)}^{t_s^2}=
\mathcal{O}_{i}^{t_s^1}\cap\mathcal{O}_{j}^{t_s^2}\ \ \ \ \ \ \ \ \ \ \ \ \ (1.1)$$

Then we also obtain
$$\{i\}=\sigma'\sigma(0)=\sigma'\sigma(\mathcal{O}_{k_s}^{t_s^5}\cap\mathcal{O}_{j_s}^{t_s^3})=\mathcal{O}_{\sigma'\sigma(k_s)}^{t_s^5}\cap\mathcal{O}_{j}^{t_s^3}\ \ \ \ \ \ \ \ \ \ \ \ \ (1.2)$$
$$\{j\}=\sigma'\sigma(j_s)=\sigma'\sigma(\mathcal{O}_{0}^{s}\cap\mathcal{O}_{k_s}^{t_s^4})=\mathcal{O}_{i}^{s}\cap\mathcal{O}_{\sigma'\sigma(k_s)}^{t_s^4}\ \ \ \ \ \ \ \ \ \ \ \ \ (1.2)$$

Which allow us to do this following computation for every $i,j\in[0,p-1]$ such that $j\in\mathcal{O}_{i}^{s}$:
\begin{eqnarray*}
f_s(e_i\otimes e_j)&=& h_s\circ\left(\id\otimes (M^*\circ M)\otimes\id\right)\left(  T_s(e_i)\otimes T_{t_s^1}(e_i)\otimes T_{t_s^2}(e_j)\otimes T_{t_s^3}(e_j) \right) \\
&=& h_s\left( T_s(e_i)\otimes   \sum_{k\in\mathcal{O}_{i}^{t_s^1}\cap\mathcal{O}_{j}^{t_s^2}}{(e_k\otimes e_k)}              \otimes T_{t_s^3}(e_j)            \right)  \\ 
&=&h_s \left(  T_s(e_i)\otimes   e_{\sigma'\sigma(k_s)}\otimes e_{\sigma'\sigma(k_s)}             \otimes T_{t_s^3}(e_j)     \right)\ \ \ \ \ \ \ \ \ \ \ \ \ \ \ \ \ \ \ \text{by (1.1)}\\ 
&=&\left(M\otimes M\right)  \left(  T_s(e_i)\otimes   T_{t_s^4}(e_{\sigma'\sigma(k_s)})\otimes T_{t_s^5}(e_{\sigma'\sigma(k_s)})           \otimes T_{t_s^3}(e_j)     \right)\\
&=&  \left(\sum_{a\in\mathcal{O}_{i}^{s}\cap\mathcal{O}_{\sigma'\sigma(k_s)}^{t_s^4}}{e_a} \right) \otimes \left(\sum_{b\in\mathcal{O}_{\sigma'\sigma(k_s)}^{t_s^5}\cap\mathcal{O}_{j}^{t_s^3}}{e_b} \right) \\ 
&=& e_j\otimes e_i \ \ \ \ \ \ \ \ \ \ \ \ \ \ \ \ \ \ \ \ \ \ \ \ \ \ \ \ \ \ \ \ \ \ \ \ \ \ \ \ \ \ \ \ \ \ \ \  \ \ \ \ \ \ \ \text{by (1.2) et (1.3)}\\
\end{eqnarray*}

The end of the proof is exactly the same of proposition \ref{NoSym}, we have $S=F\in C_{A(X)}$.
 \end{proof}

\underline{Application to the graph $\mathrm{Pr}(C_6)$}:

 Here is a particular application to the theorem  \ref{NoSymG} in a case non treated with method of Banica and Bichon in \cite{BaBi11}: $X=\mathrm{Pr}(C_6)=K_2\square C_6$.

 $$
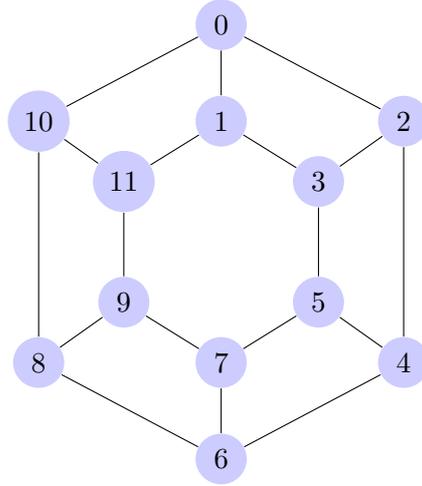
\begin{tikzpicture}
  [scale=.8,auto=left,every node/.style={circle,fill=blue!20}]
  \node (n1) at (0,3.6) {0};
  \node (n2) at (0,2) {1};
  \node (n3) at (3,2) {2};
  \node (n4) at (1.6,1)  {3};
  \node (n5) at (3,-2)  {4};
  \node (n6) at (1.6,-1) {5};
  \node (n7) at (0,-3.6) {6};
  \node (n8) at (0,-2)  {7};
  \node (n9) at (-3,-2)  {8};
  \node (n10) at (-1.6,-1)  {9};
  \node (n11) at (-3,2)  {10};
  \node (n12) at (-1.6,1)  {11};

 \foreach \from/\to in {n2/n4,n6/n4,n6/n8,n8/n10,n10/n12,n12/n2,n1/n3,n3/n5,n5/n7,n7/n9,n9/n11,n11/n1,n1/n2,n7/n8,n11/n12,n9/n10,n5/n6,n3/n4}
    \draw (\from) -- (\to);

\end{tikzpicture}$$

 First remark that $\mathrm{Pr}(C_6)$ is $\mathcal{B}$-clos. Orbitals are as follow:\medskip
 
 \ \ \ \ \ \ \begin{tabular}{ | c | c | c | c | c | c | c | c |}
 \hline                 
   $i$ & $\mathcal{O}_{i}^{1}$ &  $\mathcal{O}_{i}^{2}$ & $\mathcal{O}_{i}^{3}$ & $\mathcal{O}_{i}^{4}$ & $\mathcal{O}_{i}^{5}$ & $\mathcal{O}_{i}^{6}$ & $\mathcal{O}_{i}^{7}$ \tabularnewline \hline
$0$ & $\{1\}$ & $\{2,10\}$ & $\{3,11\}$ & $\{5,9\}$ & $\{4,8\}$ & $\{6\}$ & $\{7\}$\tabularnewline \hline
   $1$ & $\{0\}$ & $\{3,11\}$ & $\{2,10\}$ & $\{4,8\}$ & $\{5,9\}$ & $\{7\}$ & $\{6\}$\tabularnewline \hline
   $2$ & $\{3\}$ & $\{0,4\}$ & $\{1,5\}$ & $\{7,11\}$ & $\{6,10\}$ & $\{8\}$ & $\{9\}$\tabularnewline \hline
   $3$ & $\{2\}$ & $\{1,5\}$ & $\{0,4\}$ & $\{6,10\}$ & $\{7,11\}$ & $\{9\}$ & $\{8\}$\tabularnewline \hline
   $4$ & $\{5\}$ & $\{2,6\}$ & $\{3,7\}$ & $\{1,9\}$ & $\{0,8\}$ & $\{10\}$ & $\{11\}$\tabularnewline \hline
   $5$ & $\{4\}$ & $\{3,7\}$ & $\{2,6\}$ & $\{0,8\}$ & $\{1,9\}$ & $\{11\}$ & $\{10\}$\tabularnewline \hline
   $6$ & $\{7\}$ & $\{4,8\}$ & $\{5,9\}$ & $\{3,11\}$ & $\{2,10\}$ & $\{0\}$ & $\{1\}$\tabularnewline \hline
   $7$ & $\{6\}$ & $\{5,9\}$ & $\{4,8\}$ & $\{2,10\}$ & $\{3,11\}$ & $\{1\}$ & $\{0\}$\tabularnewline \hline
   $8$ & $\{9\}$ & $\{6,10\}$ & $\{7,11\}$ & $\{1,5\}$ & $\{0,4\}$ & $\{2\}$ & $\{3\}$\tabularnewline \hline
   $9$ & $\{8\}$ & $\{7,11\}$ & $\{6,10\}$ & $\{0,4\}$ & $\{1,5\}$ & $\{3\}$ & $\{2\}$\tabularnewline \hline
   $10$ & $\{11\}$ & $\{0,8\}$ & $\{1,9\}$ & $\{3,7\}$ & $\{2,6\}$ & $\{4\}$ & $\{5\}$\tabularnewline \hline
   $11$ & $\{10\}$ & $\{1,9\}$ & $\{0,8\}$ & $\{2,6\}$ & $\{3,7\}$ & $\{5\}$ & $\{4\}$ \\
 \hline  
 \end{tabular}\medskip

 The following  board show, for every value of $s\in[1,7]$, the existence of the $t_s^i$ such that assumptions of theorem \ref{NoSymG} be satisfied.\medskip

 \ \ \ \ \ \ \begin{tabular}{ | c | c | c | c | c |}
 \hline                
   $s$ & $j_s$ & $\mathcal{O}_{0}^{t_s^1}\cap\mathcal{O}_{j_s}^{t_s^2}=\{k_s\}$   & $\mathcal{O}_{0}^{s}\cap\mathcal{O}_{k_s}^{t_s^4}=\{j_s\}$ & $\mathcal{O}_{k_s}^{t_s^5}\cap\mathcal{O}_{j_s}^{t_s^3}=\{0\}$  \tabularnewline \hline
 
$1$ & $\{1\}$ & $\mathcal{O}_{0}^{1}\cap\mathcal{O}_{1}^{0}=\{1\}$ & $\mathcal{O}_{0}^{1}\cap\mathcal{O}_{1}^{0}=\{1\}$ & $\mathcal{O}_{1}^{1}\cap\mathcal{O}_{1}^{1}=\{0\}$ \tabularnewline \hline
   
$2$ & $\{2\}$ & $\mathcal{O}_{0}^{2}\cap\mathcal{O}_{2}^{5}=\{10\}$ & $\mathcal{O}_{0}^{2}\cap\mathcal{O}_{10}^{5}=\{2\}$ & $\mathcal{O}_{2}^{2}\cap\mathcal{O}_{10}^{2}=\{0\}$ \tabularnewline \hline
   
$3$ & $\{3\}$ & $\mathcal{O}_{0}^{3}\cap\mathcal{O}_{3}^{5}=\{11\}$ & $\mathcal{O}_{0}^{3}\cap\mathcal{O}_{11}^{5}=\{3\}$ & $\mathcal{O}_{3}^{3}\cap\mathcal{O}_{11}^{3}=\{0\}$ \tabularnewline \hline
   
$4$ & $\{5\}$ & $\mathcal{O}_{0}^{4}\cap\mathcal{O}_{5}^{5}=\{9\}$ & $\mathcal{O}_{0}^{4}\cap\mathcal{O}_{9}^{5}=\{5\}$ & $\mathcal{O}_{5}^{4}\cap\mathcal{O}_{9}^{4}=\{0\}$ \tabularnewline \hline
   
$5$ & $\{4\}$ & $\mathcal{O}_{0}^{5}\cap\mathcal{O}_{4}^{5}=\{8\}$ & $\mathcal{O}_{0}^{5}\cap\mathcal{O}_{8}^{5}=\{4\}$ & $\mathcal{O}_{4}^{5}\cap\mathcal{O}_{8}^{5}=\{0\}$ \tabularnewline \hline
   
$6$ & $\{6\}$ & $\mathcal{O}_{0}^{6}\cap\mathcal{O}_{6}^{0}=\{6\}$ & $\mathcal{O}_{0}^{6}\cap\mathcal{O}_{6}^{0}=\{6\}$ & $\mathcal{O}_{6}^{6}\cap\mathcal{O}_{6}^{6}=\{0\}$ \tabularnewline \hline
   
$7$ & $\{7\}$ & $\mathcal{O}_{0}^{7}\cap\mathcal{O}_{7}^{0}=\{7\}$ & $\mathcal{O}_{0}^{7}\cap\mathcal{O}_{7}^{0}=\{7\}$ & $\mathcal{O}_{7}^{7}\cap\mathcal{O}_{7}^{7}=\{0\}$ \\
 \hline  
 \end{tabular}
 \medskip

It does proves that  $\mathrm{Pr}(C_6)$ has no quantum symmetry by theorem \ref{NoSymG}.


\medskip


\medskip



\begin{thebibliography}{PRA}


\bibitem[1]{Al}B. Alspach, \emph{Point-symmetric graphs and digraphs of prime order and transitive permutation groups of prime degree}, J. Combinatorial Theory Ser. B \textbf{15} (1973), 12-17.\smallskip


\bibitem[2]{Sygeco}T. Banica, \emph{Symmetries of a generic coaction}, Math. Ann. \textbf{314} (1999), no. 4, 763-780.\smallskip

\bibitem[3]{BaHomog}T. Banica, \emph{Quantum automorphism groups of homogeneous graphs}, J. Funct. Anal. \textbf{224} (2005), no. 2, 243-280.\smallskip







\bibitem[4]{BaBi2}T. Banica, J. Bichon and G. Chenevier, \emph{Graphs having no quantum symmetry}, Ann. Inst. Fourier (Grenoble), \textbf{57} (2007), no. 3, 955-971.\smallskip

\bibitem[5]{BaBi11}T. Banica and J. Bichon, \emph{Quantum automorphism groups of vertex-transitive graphs of order $\le 11$}, J. Algebraic Combin. \textbf{26} (2007), no. 1, 83-105.\smallskip

\bibitem[6]{BaBi3}T. Banica and J. Bichon, \emph{Quantum groups acting on 4 points}, J. Reine Angew. Math. \textbf{626} (2009), 74-114.\smallskip



\bibitem[7]{bc}T. Banica and B. Collins, \emph{Integration over compact quantum groups}, Publ. Res. Inst. Math. Sci. \textbf{43} (2007), no. 2, 277-302.\smallskip










\bibitem[8]{bs}T. Banica and R. Speicher, \emph{Liberation of orthogonal Lie groups}, Adv. Math. \textbf{222} (2009), no. 4, 1461-1501. \smallskip




\bibitem[9]{BiNew}J. Bichon, \emph{Quantum automorphism groups of finite graphs}, Proc. Amer. Math. Soc. \textbf{131} (2003), no. 3, 665-673.\smallskip














 






















\bibitem[10]{KliSch}A. Klimyk and K. Schm\"{u}dgen, \emph{Quantum groups and their representations}, Texts and Monographs in Physics, Springer-Verlag, Berlin, 1997. xx+\textbf{552} pp. ISBN: 3-540-63452-5.\smallskip




\bibitem[11]{ks} C. K\"ostler and R. Speicher,  \emph{A noncommutative de Finetti theorem: invariance under quantum permutations is equivalent to freeness with amalgamation}, Comm. Math. Phys. \textbf{291} (2009), no. 2, 473-490.\smallskip


\bibitem[12]{leta}F. Lemeux and P. Tarrago,
\emph{Free wreath product quantum groups : the monoidal category, approximation properties and free probability},
Preprint arXiv:1411.4124. \smallskip



\bibitem[13]{David}M. Lupini, L. Mancinska and D. E. Roberson, \emph{Nonlocal Games and Quantum Permutation Groups}, arXiv: 1712.01820v2\smallskip






\bibitem[14]{JoyMorris}J. Morris, \emph{Automorphism Groups of Circulant Graphs- a survey}, Graph theory in Paris, 311-325, Trends Math., Birkhäuser, Basel, 2007.\smallskip

\bibitem[15]{NeTu}S. Neshveyev and L. Tuset, \emph{Compact quantum groups and their representation categories}, Cours Sp\'{e}cialis\'{e}s, 20. Soci\'{e}t\'{e} Math\'{e}matiques de France, Paris, 2013. vi+\textbf{169} pp. ISBN: 978-2-85629-777-3.\smallskip

\bibitem[16]{Podl}P. Podles, \emph{Symmetries of quantum spaces. Subgroups and quotient spaces of quantum $SU(2)$ and $SO(3)$ groups}, Comm. Math. Phys. \textbf{170} (1995), no.1, 1-20.\smallskip



\bibitem[17]{Sim}S. Schmidt, \emph{The Petersen graph has no quantum symmetry}, Bull. Lond. Math. Soc., \textbf{50} (3), 395-400, 2018.\smallskip


















\bibitem[18]{Wang}S. Wang, \emph{Quantum symmetry groups of finite spaces}, Comm. Math. Phys. \textbf{195} (1998), no. 1, 195-211.\smallskip

\bibitem[19]{Woro1}S.L. Woronowicz, \emph{Compact matrix pseudogroups}, Comm. Math. Phys. \textbf{111} (1987), no. 4, 613-665.\smallskip

\bibitem[20]{Woro2}S.L. Woronowicz, \emph{Compact quantum groups. Sym\'{e}tries quantiques}, (Les Houches, 1995), 845-884, North Holland, Amsterdam, 1998.\smallskip

\bibitem[21]{Woro3}S.L. Woronowicz, \emph{Tannaka-Krein duality for compact matrix pseudogroups. Twisted $SU(N)$ groups}, Invent. Math. \textbf{93} (1988), no. 1, 35-76.







\end{thebibliography}
\end{document}